\documentclass[12pt,draft]{article}

\usepackage{graphicx}

\usepackage{bbm}

\usepackage{amsmath}
\usepackage{amsthm}
\usepackage{amsfonts}
\usepackage{amssymb}
\usepackage{xcolor}
\usepackage{comment}

\newcommand{\eee}{{\rm e}}

\newcommand{\me}{\mathbb{E}}
\newcommand{\mn}{\mathbb{N}}
\newcommand{\mr}{\mathbb{R}}

\DeclareMathOperator{\1}{\mathbbm{1}}

\newcommand{\var}{{\rm Var \,}}
\newcommand{\cov}{{\rm Cov \,}}

\newcommand{\mmp}{\mathbb{P}}

\newtheorem{thm}{Theorem}[section]
\newtheorem{lemma}[thm]{Lemma}

\newtheorem{assertion}[thm]{Proposition}
\theoremstyle{definition}

\theoremstyle{remark}
\newtheorem{rem}[thm]{Remark}

\begin{document}
\title{A law of the iterated logarithm for
small counts in Karlin's occupancy scheme }\date{}
\author{Alexander Iksanov\footnote{Faculty of Computer Science and Cybernetics, Taras Shevchenko National University of Kyiv, Ukraine; e-mail address:
iksan@univ.kiev.ua} \ \ and \ \ Valeriya Kotelnikova\footnote{Faculty of Computer Science and Cybernetics, Taras Shevchenko National University of Kyiv, Ukraine; e-mail address: valeria.kotelnikova@unicyb.kiev.ua}}
\maketitle

\begin{abstract}
In the Karlin infinite occupancy scheme, balls are thrown independently into an infinite array of boxes $1$, $2,\ldots$, with probability $p_k$ of hitting the box $k$.
For $j,n\in\mathbb{N}$, denote by $\mathcal{K}^*_j(n)$ the number of boxes containing exactly $j$ balls provided that $n$ balls have been thrown. We call {\it small counts} the variables $\mathcal{K}^*_j(n)$, with $j$ fixed.
Our main result is a law of the iterated logarithm (LIL) for the small counts as the number of balls thrown becomes large. Its proof exploits a Poissonization technique and is based on a new LIL for infinite sums of independent indicators $\sum_{k\geq 1}\1_{A_k(t)}$ as $t\to\infty$, where the family of events $(A_k(t))_{t\geq 0}$ is not necessarily monotone in $t$. The latter LIL is an extension of a LIL obtained recently by  Buraczewski, Iksanov and Kotelnikova (2023+) in the situation that $(A_k(t))_{t\geq 0}$ forms a nondecreasing family of events.
\end{abstract}


\hphantom{2020 Mathematics Subject Classification: } Secondary: 60C05

\section{Introduction}

\subsection{Definition of the model}  \label{karlin}

Let $(p_k)_{k\in\mn}$ be a discrete probability distribution, with $p_k>0$ for infinitely many $k$. The {\it infinite occupancy scheme} is defined by independent allocation of balls over an infinite array boxes $1$, $2,\ldots$, with probability $p_k$ of hitting the box $k$. The scheme is usually called the {\it Karlin occupancy scheme} because of Karlin's seminal work \cite{Karlin:1967}. A survey of the literature on the infinite occupancy up to 2007 is given in \cite{Gnedin+Hansen+Pitman:2007}. An incomplete list of very recent contributions includes \cite{Blasi+Mena+Prunster:2022, Buraczewski+Iksanov+Kotelnikova:2023+, Chang+Grabchak:2023, Derbazi+Gnedin+Marynych:2023}. Among other things, the authors of \cite{Gnedin+Hansen+Pitman:2007} discuss applications of the scheme to ecology, database query optimization and literature. Another portion of possible applications can be found in Section 1.1 of \cite{Grabchak+Kelbert+Paris:2020}.

There are deterministic and Poissonized versions of Karlin's occupancy scheme. In a {\it deterministic version}
the $n$th ball is thrown at time~$n\in\mn$.
For $j,n\in\mn$, denote by $\mathcal{K}_j(n)$ and $\mathcal{K}_j^\ast(n)$ the number of boxes hit by at least $j$ balls and exactly $j$ balls, respectively, up to and including time $n$.
Observe that $\mathcal{K}_1(n)$ is the number of occupied boxes at time $n$. Sometimes the variables $\mathcal{K}_j^\ast(n)$, with $j$ fixed, are referred to as {\it small counts}.

To define the other version of the scheme we need an additional notation.
Let $(S_k)_{k\in\mn}$ denote a random walk with independent jumps having an exponential distribution of unit mean. The counting process $\pi:=(\pi(t))_{t\ge 0}$ given by
$\pi(t):=\#\{k\in\mn: S_k\le t\}$ for $t\geq 0$
is a Poisson process on $[0,\infty)$ of unit intensity.

In a {\it Poissonized version} of Karlin's occupancy scheme the $n$th ball is thrown at time~$S_n$, $n\in\mn$, and it is assumed that the allocation process is independent of $(S_k)_{k\in\mn}$, hence of $\pi$. Thus, in the time interval $[0,t]$ there are $\pi(t)$ balls thrown in the Poissonized version and $\lfloor t\rfloor$ balls thrown in the deterministic version.
While the occupancy counts of distinct boxes are dependent in the deterministic version, these are independent in the Poissonized version. The latter fact is a principal advantage of the Poissonized version. It is justified by the thinning property of Poisson processes.
For $j\in\mn$ and $t\ge 0$, denote by
$K_j(t)$ and $K_j^*(t)$ the number of boxes containing at least $j$ balls and exactly $j$ balls, respectively, in the Poissonized scheme at time $t$.
The random variables $$K_j(t)=\sum_{k\geq 1}\1_{\{\text{the box}~ k~\text{contains at least}~ j~\text{balls}~\text{at time}~t\}}$$ and
\begin{equation}\label{eq:kjast}
	K_j^*(t)=\sum_{k\geq 1}\1_{\{\text{the box}~ k~\text{contains exactly}~ j~\text{balls}~\text{at time}~t\}}
\end{equation}
are the infinite sums of independent indicators. As a consequence, their analysis is much simpler than that of $\mathcal{K}_j(n)$ and $\mathcal{K}_j^\ast(n)$ which are infinite sums of dependent indicators.

\subsection{Main results}

Put $$\rho(t):=\#\{k\in\mn: 1/p_k\le t\},\quad t>0$$ and note that $\rho(t)=0$ for $t\in(0,1]$. Following Karlin \cite{Karlin:1967} we assume that $\rho$ varies regularly at $\infty$ of index $\alpha\in [0,1]$, that is, $\rho(t)\sim t^\alpha L(t)$ as $t\to\infty$ for some $L$ slowly varying at $\infty$. An encyclopaedic treatment of slowly and regularly varying functions can be found in Section 1 of \cite{Bingham+Goldie+Teugels:1989}.

The function $\rho$ is said to belong to the {\it de Haan class} $\Pi$ if,
for all $\lambda>0$,
\begin{equation}\label{eq:deHaan}
	\lim_{t\to\infty}\frac{\rho(\lambda t)-\rho(t)}{\ell(t)}=\log \lambda
\end{equation}
for some $\ell$ slowly varying at $\infty$. The function $\ell$ is called {\it auxiliary}. According to Theorem~3.7.4 in \cite{Bingham+Goldie+Teugels:1989}, the class $\Pi$ is a subclass of the class of slowly varying functions. Further detailed information regarding the class $\Pi$ is given in Section~3 of \cite{Bingham+Goldie+Teugels:1989} and in \cite{Geluk+deHaan:1987}. Denote by $\Pi_{\ell,\,\infty}$ the subclass of the de Haan class $\Pi$ with the auxiliary functions $\ell$ satisfying $\lim_{t\to\infty}\ell(t)=\infty$.

In the case $\alpha\in (0,1]$, according to Theorems 3, 5 and 5' in \cite{Karlin:1967}, both $K_j^*(t)$ and $\mathcal{K}_j^\ast(n)$, centered by their means and normalized by their standard deviations, converge in distribution to a random variable with the standard normal distribution. In the case $\rho\in \Pi_{\ell,\,\infty}$, Corollary 1.6 in \cite{Iksanov+Kotelnikova:2022} provides functional central limit theorems for $K_j^*(t)$ and $\mathcal{K}_j^\ast(n)$, properly scaled. Our purpose is to prove laws of the iterated logarithm (LILs) for $K_j^*(t)$ as $t\to\infty$ and $\mathcal{K}_j^\ast(n)$ as $n\to\infty$. While doing so, we treat the three cases separately: $\alpha=0$, $\alpha\in (0,1)$ and $\alpha=1$. The reason is that the forms of the LILs are slightly or essentially different in these cases. If $\rho$ is slowly varying at $\infty$ and satisfies an additional assumption, then the actual limit relation is either a law of the single logarithm or a LIL. However, to keep the presentation simple we prefer to call LILs all the limit relations involving upper or lower limits which appear in the paper.

In Theorems \ref{thm:Karlin0}, \ref{thm:Karlin} and \ref{thm:Karlin1} we present LILs for the Poissonized variables $K_j^*(t)$ as $t\to\infty$. Theorem \ref{thm:Karlin0} covers a subcase of the case $\alpha=0$ in which $\rho\in\Pi_{\ell,\,\infty}$ with particular~$\ell$.

\begin{thm}\label{thm:Karlin0}
	Assume that \eqref{eq:deHaan} holds. If $\ell$ in \eqref{eq:deHaan} satisfies
	\begin{equation}\label{eq:slowly}
		\ell(t)~\sim~ (\log t)^\beta l(\log t),\quad t\to\infty
	\end{equation}
	for some $\beta>0$ and $l$ slowly varying at $\infty$, then, for each $j\in\mn$,
	\begin{equation}\label{eq:LILkar}
		\limsup_{t\to\infty}\frac{K^*_j(t)-\me K^*_j(t)}{({\rm Var}\,K^*_j(t)\log {\rm Var}\,K^*_j(t))^{1/2}}=\Big(\frac{2}{\beta}\Big)^{1/2}\quad\text{{\rm a.s.}}
	\end{equation}
	and
	\begin{equation}\label{eq:inf01}
		\liminf_{t\to\infty}\frac{K^*_j(t)-\me K^*_j(t)}{({\rm Var}\,K^*_j(t)\log {\rm Var}\,K^*_j(t))^{1/2}}=-\Big(\frac{2}{\beta}\Big)^{1/2}\quad\text{{\rm a.s.}}
	\end{equation}
	If $\ell$ in \eqref{eq:deHaan} satisfies
	\begin{equation}\label{eq:slowly2}
		\ell(t)~\sim~\exp(\sigma(\log t)^\lambda),\quad t\to\infty
	\end{equation}
	for some $\sigma>0$ and $\lambda\in (0,1)$, then, for each $j\in\mn$,
	\begin{equation}\label{eq:LILkar1}
		\limsup_{t\to\infty}\frac{K^*_j(t)-\me K^*_j(t)}{({\rm Var}\,K^*_j(t)\log\log {\rm Var}\,K^*_j(t))^{1/2}}=\Big(\frac{2}{\lambda}\Big)^{1/2}\quad\text{{\rm a.s.}}
	\end{equation}
	and
	\begin{equation}\label{eq:inf02}
		\liminf_{t\to\infty}\frac{K^*_j(t)-\me K^*_j(t)}{({\rm Var}\,K^*_j(t)\log\log {\rm Var}\,K^*_j(t))^{1/2}}=-\Big(\frac{2}{\lambda}\Big)^{1/2}\quad\text{{\rm a.s.}}
	\end{equation}
	In both cases
	\begin{equation}\label{eq:meK0}
		\me K^*_j(t)~\sim~ \frac{\ell(t)}{j},\quad t\to\infty
	\end{equation}
	and
	\begin{equation}\label{eq:varK0_1}
		{\rm Var}\,K^*_j(t)~\sim~ \Big(\frac{1}{j}- \frac{(2j-1)!}{(j!)^2 2^{2j}}\Big)\ell(t),\quad t\to\infty.
	\end{equation}
\end{thm}
\begin{rem}
	Treatment of the situations in which $\rho$ is slowly varying at $\infty$, yet $\rho\notin \Pi$ is beyond our reach. To reveal complications arising in this case we only mention that even the large-time asymptotics of $t\mapsto\var K^\ast_j(t)$ is not known. To find the asymptotic, a second-order relation for $\rho$ like \eqref{eq:deHaan} seems to be indispensable. If $\alpha\in (0,1]$, then the regular variation of $\rho$ alone ensures that, for all $\lambda>0$, $$\lim_{t\to\infty}\frac{\rho(\lambda t)-\rho(t)}{\rho(t)}=\lambda^\alpha-1.$$ Thus, no extra conditions are needed in this case.
\end{rem}

\begin{rem}
	Our present proof only works provided that, for some $a>0$, $\rho(t)=O((\ell(t))^a)$ as $t\to\infty$. In view of this,
	Theorem \ref{thm:Karlin0} does not cover
	the diverging slowly varying functions $\ell$ which grow slower than any positive power of the logarithm,
	for instance, $\ell(t)\sim\log\log t$ as $t\to\infty$. Indeed, it can be checked that $\lim_{t\to\infty}\ell(t)=\infty$ entails $\lim_{t\to\infty}(\rho(t)/\log t)=\infty$, whence trivially, for all $a>0$,
	$\lim_{t\to\infty}(\rho(t)/(\ell(t))^a)=\infty$.
\end{rem}

The following results are concerned with the cases $\alpha\in(0,1)$ and $\alpha=1$, respectively.

\begin{thm}\label{thm:Karlin}
	Assume that, for some $\alpha\in (0,1)$ and some $L$ slowly varying at $+\infty$,
	\begin{equation*}
		\rho(t)~\sim~ t^\alpha L(t),\quad t\to\infty.
	\end{equation*}
	Then, for each $j\in\mn$,
	\begin{equation}\label{eq:LILkar2}
		\limsup_{t\to\infty}\frac{K^*_j(t)-\me K^*_j(t)}{({\rm Var}\,K^*_j(t)\log\log {\rm Var}\,K^*_j(t))^{1/2}}=2^{1/2}\quad\text{{\rm a.s.}}
	\end{equation}
	and \begin{equation}\label{eq:inf1}
		\liminf_{t\to\infty}\frac{K^*_j(t)-\me K^*_j(t)}{({\rm Var}\,K^*_j(t)\log\log {\rm Var}\,K^*_j(t))^{1/2}}=-2^{1/2}\quad\text{{\rm a.s.}},
	\end{equation}
	\begin{equation}\label{eq:meKgeneral}
		\me K^*_j(t)~\sim~ \alpha\frac{\Gamma(j-\alpha)}{j!}t^\alpha L(t)
	\end{equation}
	and
	\begin{equation}\label{eq:varKgeneral}
		{\rm Var}\,K^*_j(t)~\sim~ c_{j,\,\alpha} t^\alpha L(t),\quad t\to\infty,
	\end{equation}
	where $\Gamma$ is the Euler gamma function and
	\begin{equation}\label{eq:cj}
		c_{j,\,\alpha}:=\alpha\Big(\frac{\Gamma(j-\alpha)}{j!}-\frac{2^{\alpha}\Gamma(2j-\alpha)}{2^{2j}(j!)^2}\Big)>0.
	\end{equation}
\end{thm}

\begin{thm}\label{thm:Karlin1}
	Assume that, for some $L$ slowly varying at $+\infty$,
	\begin{equation*}\label{eqassump1}
		\rho(t)~\sim~ tL(t),\quad t\to\infty.
	\end{equation*}
	Then, for each $j\geq 2$, relation \eqref{eq:LILkar2} holds,
	\begin{equation}\label{eq:meK1}
		\me K^*_j(t)~\sim~ \frac{1}{j(j-1)}tL(t)
	\end{equation}
	and
	\begin{equation}\label{eq:cj1}
		\lim_{t\to\infty}\frac{{\rm Var}\,K^*_j(t)}{tL(t)}= \frac{1}{j(j-1)}-\frac{(2j-2)!}{2^{2j-1}(j!)^2}=c_{j,\,1}.
	\end{equation}
	
	\noindent Assume
	that, for each small enough $\gamma>0$,
	\begin{equation}\label{eq:exotic}
		\lim_{n\to\infty}\frac{\hat L(\exp((n+1)^{1+\gamma}))}{\hat L(\exp(n^{1+\gamma}))}=0,
	\end{equation}
	where $\hat L(t):=\int_t^\infty y^{-1}L(y){\rm d}y$, being
	well-defined for large $t$, is a function slowly varying at $\infty$ and satisfying
	\begin{equation}\label{eq:Lhat}
		\lim_{t\to\infty} \frac{L(t)}{\hat{L}(t)}=0.
	\end{equation}
	Then relation \eqref{eq:LILkar2} holds with $j=1$. If \eqref{eq:exotic} does not hold, then
	\begin{equation}\label{eq:LILkar200}
		\limsup_{t\to\infty}\frac{K^*_1(t)-\me K^*_1(t)}{({\rm Var}\,K^*_1(t)\log\log {\rm Var}\,K^*_1(t))^{1/2}}\leq 2^{1/2}\quad\text{{\rm a.s.}}
	\end{equation}
	and
	\begin{equation}\label{eq:inf11}
		\liminf_{t\to\infty}\frac{K^*_1(t)-\me K^*_1(t)}{({\rm Var}\,K^*_1(t)\log\log {\rm Var}\,K^*_1(t))^{1/2}}\geq -2^{1/2}\quad\text{{\rm a.s.}}
	\end{equation}
	In any event
	\begin{equation}\label{eq:exo}
		{\rm Var}\,K^*_1(t)~\sim~ \me K^*_1(t)~\sim~ t \hat L(t),\quad t\to\infty.
	\end{equation}
\end{thm}

Theorems \ref{thm:Karlin0}, \ref{thm:Karlin} and \ref{thm:Karlin1} will be deduced in Section \ref{sec:Karlin} from the LIL for infinite sums of independent indicators given in Theorem~\ref{thm:main}.

Finally, we present LILs for the variables $\mathcal{K}^*_j(n)$.

\begin{thm}\label{thm:depoiss}
	Under the assumptions of Theorems \ref{thm:Karlin0}, \ref{thm:Karlin} or \ref{thm:Karlin1}, for $j\in\mn$, all the LILs stated there
	hold true with $\mathcal{K}^*_j(n)$, $\me \mathcal{K}^*_j(n)$ and $\var \mathcal{K}^*_j(n)$ replacing $K_j^*(t)$, $\me K_j^*(t)$ and $\var K_j^*(t)$, and $n\to\infty$ replacing $t\to\infty$.
\end{thm}

A transfer of results available for the Poissonized version to the deterministic version is called
{\it de-Poissonization}. Theorem \ref{thm:depoiss} will be deduced in Section~\ref{sec:Karlin} from Theorems~\ref{thm:Karlin0},~\ref{thm:Karlin} and \ref{thm:Karlin1} with the help of a de-Poissonization technique.

\section{LIL for infinite sums of independent indicators}

Let $(A_1(t))_{t\geq 0}$, $(A_2(t))_{t\geq 0},\ldots$ be independent families of events defined on a common probability space $(\Omega, \mathcal{F}, \mmp)$. Assume that $\sum_{k\geq 1}\mmp (A_k(t))<\infty$, for each $t\geq 0$, and then put $$X(t):=\sum_{k\geq 1}\1_{A_k(t)},\quad t\geq 0.$$ Since, for $t\geq 0$, $b(t):=\me X(t)=\sum_{k\geq 1}\mmp (A_k(t))<\infty$, we infer $X(t)<\infty$ almost surely (a.s.) and further $$a(t):={\rm Var}\,X(t)=\sum_{k\geq 1}\mmp(A_k(t))(1-\mmp(A_k(t)))\leq b(t)<\infty.$$

Under the assumption that, for each $k\in\mn$ and $0\leq s<t$, $A_k(s)\subseteq A_k(t)$ a LIL for $X(t)$ can be found in Theorem 1.6 of \cite{Buraczewski+Iksanov+Kotelnikova:2023+}. As an application, LILs for $K_j(t)$ were proved in that paper, see Theorems 3.1, 3.3 and 3.4 therein. According to \eqref{eq:kjast}, the variable $K^\ast_j(t)$ is a particular instance of $X(t)$. However, for each $k\in\mn$, the corresponding events $(A_k(t))_{t\geq 0}$ are not monotone in~$t$, which shows that a LIL for $K_j^\ast(t)$ cannot be deduced from Theorem 1.6 of \cite{Buraczewski+Iksanov+Kotelnikova:2023+}. This serves a motivation for the present section. Here, dropping the monotonicity assumption we provide sufficient conditions under which a LIL for $X(t)$ holds.

We shall prove a LIL for $X(t)$ under the following assumptions (A1)-(A5) and (B1)-(B21) or (B22). The lack of monotonicity only affects our proof of $\limsup_{t\to\infty}\leq 1$ a.s. to be done under (A1)-(A5). In view of this, (A2)-(A5) are modified versions of the corresponding assumptions in \cite{Buraczewski+Iksanov+Kotelnikova:2023+}. (B1), (B21) and (B22) coincide with the corresponding assumptions in \cite{Buraczewski+Iksanov+Kotelnikova:2023+} under which the relation $\limsup_{t\to\infty}\geq 1$ a.s. was proved in the cited article.

\noindent (A1) $\lim_{t\to\infty}a(t)=\infty$.

\sloppy \noindent (A2) There exist independent a.s. nondecreasing stochastic processes $(\Phi_1(t))_{t\geq 0}$, $(\Phi_2(t))_{t\geq 0},\ldots$
taking values in $\{0, 1, 2, \ldots, M\}$ for some $M\in\mn$ and satisfying

\noindent (a) for each $k\in\mn$, $0\le s< t$, $|\1_{A_k(t)}-\1_{A_k(s)}|\le \Phi_k(t) - \Phi_k(s)$ a.s.;

\noindent (b) for each $t\geq 0$, $f(t):=\me Y(t)<\infty$, where $Y(t):=\sum_{k\geq 1}\Phi_k(t)$ for $t\ge 0$;

\noindent (c) $b(0)\le f(0)$.

\begin{rem}\label{rem:bf}
	(A2a) and (A2b) entail, for $0\le s<t$,
	\begin{equation}\label{eq:bf_with_s}
		|b(t)-b(s)|\le \me \sum_{k\geq 1} |\1_{A_k(t)}-\1_{A_k(s)}|\le f(t) - f(s).
	\end{equation}
	Inequality \eqref{eq:bf_with_s} with $s=0$ and (A2c) together imply that
	\begin{equation}\label{eq:bf}
		b(t)\le f(t) - f(0)+b(0)\le f(t),\quad t\geq 0.
	\end{equation}
\end{rem}

\begin{rem}\label{rem:A2_ex}
	Here is an example of $X$ satisfying (A2) which is motivated by a prospective application of the LIL for $X(t)$ to the variables $K_j^\ast(t)$. Let $(B_1(t))_{t\ge 0}$, $(B_2(t))_{t\ge 0},\ldots$ and $(C_1(t))_{t\ge 0}$, $(C_2(t))_{t\ge 0},\ldots$ be two families of independent events satisfying
	
	\noindent (i) for each $k\in\mn$ and $t\ge 0$, $C_k(t)\subseteq B_k(t)$;
	
	\noindent (ii) for each $k\in\mn$ and $0\leq s<t$, $B_k(s)\subseteq B_k(t)$ and $C_k(s)\subseteq C_k(t)$;
	
	\noindent (iii) for $t\geq 0$, $\sum_{k\geq 1} \mmp(B_k(t))<\infty$.
	
	\noindent For each $k\in\mn$ and $t\ge 0$, put $A_k(t):=B_k(t)\setminus C_k(t)$ and $\Phi_k(t):=\1_{B_k(t)}+\1_{C_k(t)}$. The so defined $\Phi_k$ is a.s.\ nondecreasing.
	Since, for $0\leq s<t$, $\1_{C_k(s)}\le\1_{C_k(t)}$ and $\1_{B_k(s)}\le\1_{B_k(t)}$ a.s. we conclude that
	$$
	|\1_{A_k(t)}-\1_{A_k(s)}|=|\1_{B_k(t)}-\1_{C_k(t)}-\1_{B_k(s)}+\1_{C_k(s)}|\le \Phi_k(t)-\Phi_k(s) \quad\text{a.s.}
	$$
	While (A2b) is a consequence of (iii), (A2c) is justified by $\mmp(A_k(0))=\mmp(B_k(0))-\mmp(C_k(0))\le \me \Phi_k(0)$.
	
	Putting $C_k(t):=\oslash$ for all $k\in\mn$ and $t\geq 0$ we recover the case of monotone in $t$ families $(A_k(t))_{t\geq 0}$ treated in \cite{Buraczewski+Iksanov+Kotelnikova:2023+}.
\end{rem}

\begin{rem}
	The assumption imposed in~\cite{Buraczewski+Iksanov+Kotelnikova:2023+} that, for each $k\in\mn$, the family $(A_k(t))_{t\geq 0}$ is 
	nondecreasing in $t$ simplifies significantly the analysis of \eqref{eq:Yf}. Indeed, for any $\theta>0$, we then infer 
	$\sup_{v\in[0,\theta]}\,|X(t+v)-X(t)|=X(t+\theta)-X(t)$ and $\sup_{v\in[0,\theta]}\,|b(t+v)-b(t)|=b(t+\theta)-b(t)$.  
	In the absence of the monotonicity assumption, it is necessary to find some monotone majorant for $|X(t+v)-X(t)|$ which is sufficiently close to the true supremum.
	
	One may expect that $f$ behaving like $f(t)=O(b(t))$ as $t\to\infty$ should do the job. 
	What is not trivial is that $f$ satisfying 
	$\lim_{t\to\infty}(f(t)/b(t))= \infty$  
	may also be suitable. For instance, consider the setting of Theorem \ref{thm:Karlin0} and $X(t):=K_j^*(t)$ for $j\in\mn$. By~\eqref{eq:meK0}, $b(t)\sim {\rm const}\, \ell(t)$ as $t\to\infty$, and by \eqref{eq:f_remark} and~\eqref{eq:mean0old}, $f(t)\sim{\rm const}\,\rho(t)$ as $t\to\infty$. Applying Lemma \ref{eq:thm21} we conclude that indeed $\lim_{t\to\infty}(f(t)/b(t))=\infty$. 
\end{rem}

\noindent (A3) Under (A2), there exists $\mu^\ast\geq 1$ such that $f(t)=O((a(t))^{\mu^\ast})$
as $t\to\infty$. In view of~\eqref{eq:bf} and $a(t)\le b(t)$ for $t\ge 0$, necessarily $\mu^\ast\geq 1$. Put
\begin{equation}\label{eq:infim}
	\mu:=\inf\{\mu^\ast: f(t)=O((a(t))^{\mu^\ast})\}.
\end{equation}
If $\mu=1$, we assume additionally that
either $f$ is eventually continuous or $${\lim\inf}_{t\to\infty}(\log f(t-1)/\log f(t))>0;$$

\noindent
and that
\begin{equation}\label{eq:infim2}
	f(t)/a(t)=O(z_q(a(t))),\quad t\to\infty,
\end{equation}
where $z_q(t):=(\log t)^{q}\mathcal{L}(\log t)$ for some $q\geq 0$ and $\mathcal{L}$ slowly varying at $\infty$ and, if $q>0$, $f(t)/a(t)\neq O(z_s(a(t)))$ for $s\in (0,q)$.

Before introducing our next assumption we need some preparation.
In view of (A1) and $a(t)\leq f(t)$ for $t\ge 0$, we infer $\lim_{t\to\infty}f(t)=\infty$. For each $\varrho\in (0,1)$, put
\begin{equation}\label{eq:mutheta}
	\mu_\varrho:=\mu+\varrho\quad \text{if}~\mu>1\quad\text{and}\quad q_\varrho:=q+\varrho \quad \text{if}~\mu=1.
\end{equation}
Assuming (A3), fix any $\kappa\in (0,1)$ and $\varrho\in (0,1)$ and put
\begin{equation}\label{tn}
	t_n=t_n(\kappa, \mu):=\inf\{t>0: f(t)>v_n(\kappa,\mu)\}
\end{equation}
for $n\in\mn$, where $v_n(\kappa, 1)=v_n(\kappa, 1, q, \varrho)=\exp(n^{(1-\kappa)/(q_\varrho+1)})$ and $v_n(\kappa,\mu)=v_n(\kappa, \mu, \varrho)=n^{\mu_\varrho(1-\kappa)/(\mu_\varrho-1)}$ for $\mu>1$. Plainly, the sequence $(t_n)_{n\in\mn}$ is nondecreasing with $\lim_{n\to\infty} t_n=+\infty$.

\noindent (A4) Fix any $\kappa\in (0,1)$ and $\varrho\in (0,1)$. There exists a function $a_0$ satisfying $a(t)\sim a_0(t)$ as $t\to\infty$, and, for each $n$ large enough, there exists $s_n=s_n(\kappa,\mu)\in[t_n(\kappa,\mu),\,t_{n+1}(\kappa,\mu)]$ such that $a_0(t)\geq a_0(s_n)$ for all $t\in [t_n, t_{n+1}]$.

\begin{rem}\label{rem:A4}
	A sufficient condition for (A4) is either
	eventual lower semi-continuity or eventual monotonicity of $a_0$. The former means that $\liminf_{y\to x}a_0(y)\geq a_0(x)$, for all large enough $x$.
\end{rem}

\noindent (A5) For each $n$ large enough, there exists $A>1$ and a partition $t_n=t_{0,\,n}<t_{1,\,n}<\ldots<t_{j,\,n}=t_{n+1}$ with $j=j_n$ satisfying $$1\leq f(t_{k,\,n})-f(t_{k-1,\,n})\leq A,\quad 1\leq k\leq j$$
and, for all $\varepsilon>0$, $\big(j_n\exp(-\varepsilon (a(s_n))^{1/2})\big)$ is a summable sequence.
\begin{rem}\label{suff}
	A sufficient condition for (A5) is that $f$ is eventually strictly increasing and eventually continuous. Indeed, one can then choose a partition that satisfies, for large~$n$, $f(t_{k,\,n})-f(t_{k-1,\,n})=1$ for $k\in\mn$, $k\leq j-1$ and $f(t_{j,\,n})-f(t_{j-1,\,n})\in [1,2)$. As a consequence,
	$$j_n=\lfloor v_{n+1}(\kappa, \mu)-v_n(\kappa,\mu)\rfloor=o(a(s_n)),\quad n\to\infty$$ by Lemma \ref{lem:aux1}(b) below, so that the sequence $\big(j_n\exp(-\varepsilon (a(s_n))^{1/2})\big)$ is indeed summable.
\end{rem}

Assuming (A1) and (A3), fix any $\gamma>0$ and put
\begin{equation}\label{eq:tau}
	\tau_n=\tau_n(\gamma,\mu):=\inf\{t>0: a(t)>w_n(\gamma,\mu)\}
\end{equation}
for large $n\in\mn$ with $\mu$ as given in \eqref{eq:infim}. Here, with $q$ as given in \eqref{eq:infim2}, $w_n(\gamma, 1)=w_n(\gamma, 1, q)=\exp(n^{(1+\gamma)/(q+1)})$ if $\mu=1$ and $w_n(\gamma, \mu)=n^{(1+\gamma)/(\mu-1)}$ if $\mu>1$.

\noindent (B1) The function $a$ is eventually continuous or $\lim_{t\to\infty}(\log a(t-1)/\log a(t))=1$ if $\mu=1$ and $\lim_{t\to\infty}(a(t-1)/a(t))=1$ if $\mu>1$.

\noindent (B21) For sufficiently large $t>0$ and each $\varsigma>0$, let $R_\varsigma(t)$ denote a set of positive integers satisfying the following two conditions: for each $\varsigma>0$ and each $\gamma>0$, both close to $0$ there exists $n_0=n_0(\varsigma, \gamma)\in\mn$ such that the sets $R_\varsigma(\tau_{n_0}(\gamma,\mu))$, $R_\varsigma(\tau_{n_0+1}(\gamma,\mu)),\ldots$ are disjoint; and
\begin{equation*}\label{eq:onevar}
	\lim_{t\to\infty}\frac{{\rm Var}\Big(\sum_{k\in R_\varsigma(t)}\1_{A_k(t)}\Big)}{{\rm Var}\,X(t)}=1-\varsigma.
\end{equation*}

\noindent (B22) For sufficiently large $t>0$, let $R_0(t)$ denote a set of positive integers satisfying the following two conditions: for each $\gamma>0$ close to $0$ there exists $n_0=n_0(\gamma)\in\mn$ such that the sets $R_0(\tau_{n_0}(\gamma,\mu))$, $R_0(\tau_{n_0+1}(\gamma,\mu)),\ldots$ are disjoint; and
\begin{equation}\label{eq:one}
	\lim_{t\to\infty}\frac{{\rm Var}\Big(\sum_{k\in R_0(t)}\1_{A_k(t)}\Big)}{{\rm Var}\,X(t)}=1.
\end{equation}

Now we are ready to present a LIL for infinite sums of independent indicators.
\begin{thm}\label{thm:main}
	Suppose (A1)-(A5), (B1) and either (B21) or (B22). Then, with $\mu\geq 1$ and $q\geq 0$ as defined in \eqref{eq:infim} and \eqref{eq:infim2}, respectively, $$\limsup_{t\to\infty}\frac{X(t)-\me X(t)}{(2(q+1){\rm Var}\,X(t)\log\log
		{\rm Var}\,X(t))^{1/2}}=1\quad\text{{\rm a.s.}}$$ and $$\liminf_{t\to\infty}\frac{X(t)-\me X(t)}{(2(q+1){\rm Var}\,X(t)\log\log
		{\rm Var}\,X(t))^{1/2}}=-1\quad\text{{\rm a.s.}}$$ if $\mu=1$ and $$\limsup_{t\to\infty}\frac{X(t)-\me X(t)}{(2(\mu-1){\rm Var}\,X(t)\log {\rm Var}\,X(t))^{1/2}}=1\quad\text{{\rm a.s.}}$$ and $$\liminf_{t\to\infty}\frac{X(t)-\me X(t)}{(2(\mu-1){\rm Var}\,X(t)\log {\rm Var}\,X(t))^{1/2}}=-1\quad\text{{\rm a.s.}}$$ if $\mu>1$.
\end{thm}

Our proof of Theorem \ref{thm:main} given in Section \ref{sec:main} is a modified version of the proof of Theorem 1.6 in \cite{Buraczewski+Iksanov+Kotelnikova:2023+}.

\section{Proof of Theorem \ref{thm:main}}

\subsection{Auxiliary results}

We start with a simple inequality which will be used in the last part of the proof of Proposition \ref{lilhalf1}.

\begin{lemma}\label{lem:exp_mom}
	Suppose (A2). Then, for $\vartheta \in\mr$ and $t> s\ge 0$,
	$$\me \exp(\vartheta (Y(t)-Y(s)))\le \exp((\eee^{\vartheta M}-1)(f(t)-f(s))).$$
\end{lemma}
\begin{proof}
	For $k\in\mn$ and $t>s\ge 0$, $\1_{\{\Phi_k(t)-\Phi_k(s)>0\}}=\1_{\{\Phi_k(t)-\Phi_k(s)\ge 1\}}\le \Phi_k(t)-\Phi_k(s)$ a.s. The equality stems from the fact that $\Phi_k$ only takes nonnegative integer values.
	Hence, for $\vartheta\in\mr$ and $0\leq s<t$,
	\begin{multline*}
		\me \exp(\vartheta (Y(t)-Y(s))=\prod_{k\geq 1}\me \exp (\vartheta (\Phi_k(t)-\Phi_k(s)))\\= \prod_{k\geq 1}\big(1+\me(\eee^{\vartheta(\Phi_k(t)-\Phi_k(s))}-1)\1_{\{\Phi_k(t)-\Phi_k(s)>0\}}\big)\le \prod_{k\geq 1}\big(1+(\eee^{\vartheta M}-1)\me\1_{\{\Phi_k(t)-\Phi_k(s)>0\}}\big)\\\leq \exp\Big((\eee^{\vartheta M}-1)\sum_{k\geq 1}\me(\Phi_k(t)-\Phi_k(s))\Big)=
		\exp((\eee^{\vartheta M}-1)(f(t)-f(s))).
	\end{multline*}
\end{proof}

For each $B\geq 0$ and each $D>1$, put
\begin{equation}\label{eq:defg}
	g_{1,\,B}(t):=(B+1)\log\log t,\quad t>\eee\quad\text{and}\quad g_D(t):=(D-1)\log t,\quad t>1.
\end{equation}
Lemma \ref{lem:aux1} does two things. First, it explains the choice of the sequences $(t_n)$ and $(v_n)$ and the functions $g_{1,\,q_\varrho}$ and $g_{\mu_\varrho}$ (even though $(t_n)$ is not present in Lemma \ref{lem:aux1} explicitly, it is of crucial importance for defining the sequence $(s_n)$.) Second, it secures a successful application of the Borel-Cantelli lemma in the proof of Proposition \ref{lilhalf1}.

\begin{lemma}\label{lem:aux1}
	Suppose (A1), (A3) and (A4). Fix any $\varrho\in (0,1)$, any $\kappa\in (0,1)$ and let $q_\varrho$ and $\mu_\varrho$ be as defined in \eqref{eq:mutheta}.
	
	\noindent (a) If $\mu$ in \eqref{eq:infim} is equal to $1$, then $\exp(-g_{1,\,q_\varrho}(a(s_n(\kappa, 1))))=O(n^{-(1-\kappa)})$ as $n\to\infty$, and if $\mu>1$, then $\exp(-g_{\mu_\varrho}(a(s_n(\kappa, \mu))))=O(n^{-(1-\kappa)})$.
	
	\noindent (b) There exists an integer $r\geq 2$ such that $\big(\big((v_{n+1}(\kappa,\mu)-v_n(\kappa, \mu))/a(s_n)\big)^r\big)$ is a summable sequence.
\end{lemma}
\begin{proof}
	
	\noindent (a) Using the definition of $t_n$, the fact that $f$ is nondecreasing and (A3), we conclude that
	\begin{equation}\label{eq:ineq10}
		\exp(n^{(1-\kappa)/(q_\varrho+1)})\leq f(t_n(\kappa, 1))\le f(s_n(\kappa, 1))=O(a(s_n(\kappa, 1))z_q(a(s_n(\kappa, 1)))),\quad n\to\infty
	\end{equation}
	and, for $\mu>1$,
	\begin{equation}\label{eq:ineq11}
		n^{\mu_\varrho(1-\kappa)/(\mu_\varrho-1)}\leq f(t_n(\kappa,\mu))\leq f(s_n(\kappa, \mu))=O((a(s_n(\kappa,\mu)))^{\mu_\varrho}),\quad n\to\infty.
	\end{equation}
	Since $\lim_{t\to\infty}(\log z_q(t)/\log t)=0$ we infer $$\exp(-g_{1,\,q_\varrho}(a(s_n(\kappa, 1))))=(\log a(s_n(\kappa, 1)))^{-(q_\varrho+1)}=O(n^{-(1-\kappa)}),\quad n\to\infty.$$ Also, for $\mu>1$, $$\exp(-g_{\mu_\varrho}(a(s_n(\kappa, \mu))))=(a(s_n(\kappa,\mu)))^{-(\mu_\varrho-1)}=O(n^{-(1-\kappa)}),\quad n\to\infty.$$
	
	\noindent (b) We start by proving that (A3) with $\mu=1$ entails
	\begin{equation}\label{eq:ineqlog}
		\log a(s_n(\kappa, 1))=O(n^{(1-\kappa)/(q_\varrho+1)}),\quad n\to\infty.
	\end{equation}
	Assume that $f$ is eventually continuous. Then $f(t_n(\kappa, 1))=v_n(\kappa, 1)$ for large enough $n$ and thereupon $\log a(s_n(\kappa, 1))\leq \log f(s_n(\kappa, 1))\le \log f(t_{n+1}(\kappa, 1))= (n+1)^{(1-\kappa)/(q_\varrho+1)}$ for large~$n$. Assuming that ${\lim\inf}_{t\to\infty}(\log f(t-1)/\log f(t))>0$ we obtain \eqref{eq:ineqlog} as a consequence of
	$\log f(t_{n+1}(\kappa, 1)-1) \leq (n+1)^{(1-\kappa)/(q_\varrho+1)}$ and $\log a(s_n(\kappa, 1))\leq \log f(s_n(\kappa, 1))\le \log f(t_{n+1}(\kappa, 1))$.
	
	We proceed by noting that, as $n\to\infty$,
	\begin{multline*}
		v_{n+1}(\kappa,1)-v_n(\kappa, 1)=\exp((n+1)^{(1-\kappa)/(q_\varrho+1)})-\exp(n^{(1-\kappa)/(q_\varrho+1)})\\~\sim~((1-\kappa)/(q_\varrho+1))n^{((1-\kappa)/(q_\varrho+1))-1}\exp(n^{(1-\kappa)/(q_\varrho+1)})
	\end{multline*}
	and, for $\mu>1$,
	\begin{multline*}
		v_{n+1}(\kappa,\mu)-v_n(\kappa, \mu)=(n+1)^{\mu_\varrho(1-\kappa)/(\mu_\varrho-1)}-n^{\mu_\varrho(1-\kappa)/(\mu_\varrho-1)}\\~\sim~ (\mu_\varrho(1-\kappa)/(\mu_\varrho-1))n^{(1-\mu_\varrho\kappa)/(\mu_\varrho-1)}.
	\end{multline*}
	Write
	\begin{multline*}
		\frac{1}{a(s_n(\kappa, 1))}=O((\log a(s_n(\kappa, 1)))^{q_\varrho}\exp(-n^{(1-\kappa)/(q_\varrho+1)}))\\=O(n^{q_\varrho(1-\kappa)/(q_\varrho+1)}\exp(-n^{(1-\kappa)/(q_\varrho+1)})),\quad n\to\infty.
	\end{multline*}
	Here, the first equality is implied by $z_q(t)=O((\log t)^{q_\varrho})$ as $t\to\infty$ and \eqref{eq:ineq10}, and the second equality is a consequence of \eqref{eq:ineqlog}. In the case $\mu>1$, invoking~\eqref{eq:ineq11} we infer
	\begin{equation*}
		\frac{1}{a(s_n(\kappa,\mu))}=O(n^{-(1-\kappa)/(\mu_\varrho-1)})),\quad n\to\infty.
	\end{equation*}
	Thus, we have proved that,
	for $\mu\geq 1$, $$\frac{v_{n+1}(\kappa,\mu)-v_n(\kappa, \mu)}{a(s_n)}=O(n^{-\kappa}),\quad n\to\infty.$$
	Choosing any integer $r\geq 2$ satisfying $r\kappa>1$ completes the proof of part~(b).
\end{proof}

For $k\in\mn$ and $t\geq 0$, put $X^\ast(t):=X(t)-\me X(t)$ and $\eta_k(t):=\1_{A_k(t)}-\mmp(A_k(t))$. Note that
\begin{equation*}
	X^\ast(t)=\sum_{k\geq 1} \eta_k(t),\quad t\geq 0
\end{equation*}
and that $\eta_1(t)$, $\eta_2(t),\ldots$ are independent centered random variables.

Lemma \ref{ineq1} provides a uniform bound for higher moments of the increments of $X^\ast$. The bound serves a starting point of the chaining argument in the spirit of Lemma \ref{lem:bor}. A result of an application of Lemma \ref{lem:bor} to the present setting is given in Lemma \ref{billappl}.

\begin{lemma}\label{ineq1}
	Suppose (A2). Let $r\in\mn$ and $t,s\geq 0$. Then
	\begin{equation}\label{eq:5}
		\me (X^\ast(t)-X^\ast(s))^{2r}\leq D_r \max(|f(t)- f(s)|^r, |f(t)-f(s)|)
	\end{equation}
	for a positive constant $D_r$ which does not depend on $t$ and $s$.
\end{lemma}
\begin{proof}
	In view of the representation
	\begin{equation*}
		X^\ast(t)-X^\ast(s)=\sum_{k\geq 1}(\1_{A_k(t)}-\mmp(A_k(t))-\1_{A_k(s)}+\mmp(A_k(s)))=:\sum_{k\geq 1}\eta_k(s,t),
	\end{equation*}
	the variable $X^\ast(t)-X^\ast(s)$ is an infinite sum of independent centered random variables with finite
	moment of order $2r$.
	
	Invoking Rosenthal's inequality (Theorem 3 in \cite{Rosenthal:1970}) in the case $r\geq 2$ we infer
	$$\me (X^\ast(t)-X^\ast(s))^{2r}\leq C_r \max\Big(\Big(\sum_{k\geq 1}\me (\eta_k(s,t))^2\Big)^r, \sum_{k\geq 1}\me (\eta_k(s,t))^{2r}\Big).$$ In the case $r=1$, the inequality trivially holds with $C_1=1$ as is seen from
	$$\me (X^\ast(t)-X^\ast(s))^2= \sum_{k\geq 1}\me (\eta_k(s,t))^2.$$
	In view of (A2), for $r\in\mn$ and $0\le s< t$,
	\begin{multline*}
		\sum_{k\geq 1} \me (\eta_k(s,t))^{2r}\le 2^{2r-1} \sum_{k\geq 1} \Big(\me (\1_{A_k(t)}-\1_{A_k(s)})^{2r}+(\mmp(A_k(t))-\mmp(A_k(s)))^{2r}\Big)\\\le 2^{2r-1} \sum_{k\in\mn} \Big(\me |\1_{A_k(t)}-\1_{A_k(s)}|+|\mmp(A_k(t))-\mmp(A_k(s))|\Big)\\\le 2^{2r}\sum_{k\geq 1} \me (\Phi_k(t)-\Phi_k(s))=2^{2r}(f(t)-f(s)).
	\end{multline*}
	Here, we have used $(a+b)^{2r}\le 2^{2r-1}(a^{2r}+b^{2r})$, $a,b\in\mr$ for the first inequality, the fact that $|\1_{A_k(t)}-\1_{A_k(s)}|\in \{0,1\}$ a.s.\ and $|\mmp(A_k(t))-\mmp(A_k(s))|\in [0,1]$ for the second and~(A2a) for the third. The argument for the case $0\leq t<s$ is analogous.
	
	Combining fragments together we conclude that \eqref{eq:5} holds with $D_r:=2^{2r}C_r$.
\end{proof}

The next result is borrowed from Lemma 2 in \cite{Longnecker+Serfling:1977}.
\begin{lemma}\label{lem:bor}
	Let $\xi_1$, $\xi_2,\ldots$ be random variables. Fix any $m\in\mn$ and assume that
	$$\me|\xi_{i+1}+\ldots+\xi_k|^{\lambda_1}\leq (u_{i+1}+\ldots+ u_k)^{\lambda_2},\quad 0\leq i<k\leq m$$
	for some $\lambda_1>0$, some $\lambda_2>1$ and some nonnegative numbers $u_1,\ldots, u_m$. Then
	$$\me(\max_{1\leq k\leq m}|\xi_1+\ldots+\xi_k|)^{\lambda_1}\leq A_{\lambda_1,\lambda_2}(u_1+\ldots+u_m)^{\lambda_2}$$
	for a positive constant $A_{\lambda_1,\lambda_2}$.
\end{lemma}

\begin{lemma}\label{billappl}
	Suppose (A2) and (A5). Then, for any integer $r\geq 2$, there exists a positive constant $A_r$ such that
	\begin{equation}\label{ineq4}
		\me (\max_{1\leq k\leq j}\,|X^\ast(t_{k-1,\,n})-X^\ast(t_n)|)^{2r}\leq A_r (v_{n+1}(\kappa,\mu)-v_n(\kappa,\mu))^r.
	\end{equation}
	Here, $j$ and $(t_{k,\,n})_{0\leq k\leq j}$ are as defined in (A5), and $v_n(\kappa,\mu)$ is as defined in~\eqref{tn}.
\end{lemma}
\begin{proof}
	We first show that the assumption of Lemma \ref{lem:bor} holds
	with $\lambda_1=2r$, $\lambda_2=r$, $m=j-1$, $\xi_k:=X^\ast(t_{k,\,n})-X^\ast(t_{k-1,\,n})$ and $u_k:=D_r^{1/r}(f(t_{k,\,n})-f(t_{k-1,\,n}))$ for $k\in\mn$, where $D_r$ is the constant defined in Lemma \ref{ineq1}. Let $0\leq i<k\leq j-1$. By (A5),
	$f(t_{k,\,n})-f(t_{i,\,n})=\sum_{l=i+1}^k (f(t_{l,\,n})-f(t_{l-1,\,n}))\geq 1$. This in combination with
	Lemma \ref{ineq1} yields
	\begin{multline*}
		\me |\xi_{i+1}+\ldots+\xi_k|^{2r}=\me(X^\ast(t_{k,\,n})-X^\ast(t_{i,\,n}))^{2r}\\\leq D_r \max((f(t_{k,\,n})-f(t_{i,\,n}))^r, f(t_{k,\,n})-f(t_{i,\,n}))=D_r (f(t_{k,\,n})-f(t_{i,\,n}))^r=\Big(\sum_{l=i+1}^k u_l\Big)^r,
	\end{multline*}
	thereby proving that the assumption of Lemma \ref{lem:bor} does indeed hold.
	Hence, inequality~\eqref{ineq4} follows from Lemma \ref{lem:bor} and
	the definition of $t_n$:
	\begin{multline*}
		\me (\max_{1\leq k\leq j}\,|X^\ast(t_{k-1,\,n})-X^\ast(t_n)|)^{2r}=\me (\max_{1\leq k\leq j-1}|\xi_1+\ldots+\xi_k|)^{2r}\leq A_{2r,\,r}\Big(\sum_{l=1}^{j-1} u_l\Big)^r\\=A_{2r,\,r} D_r (f(t_{j-1,\,n})-f(t_n))^r\leq A_{2r,\,r} D_r (v_{n+1}(\kappa,\mu)-v_n(\kappa,\mu))^r.
	\end{multline*}
\end{proof}

\subsection{Proof of Theorem \ref{thm:main}}\label{sec:main}

We start with a lemma and a proposition which are in essence Lemma 4.13 and Proposition 4.7 in \cite{Buraczewski+Iksanov+Kotelnikova:2023+}. Although our present assumption (A3) is slightly different from the corresponding assumption in \cite{Buraczewski+Iksanov+Kotelnikova:2023+}, we have checked that the proofs of the aforementioned results in \cite{Buraczewski+Iksanov+Kotelnikova:2023+} go through.
\begin{lemma}\label{lem:new}
	Suppose (A1), (A3), (B1) and either (B21) or (B22) and let $\mu\geq 1$ be as given in \eqref{eq:infim}. For sufficiently small $\delta >0$, pick $\gamma\in (0, (\sqrt{5}-1)/2) >0$ satisfying $(1+\gamma)(1-\delta^2/8)<1$. Then
	$$\limsup_{n\to \infty} (\liminf_{n\to\infty}) \frac 1{(2a(\nu_n )h_0(a(\nu_n))^{1/2}} \sum_{k\ge 1} \eta_k(\nu_n ) \ge 1-\delta~(\leq -(1-\delta)) \quad\text{{\rm a.s.}},$$ where $\nu_n$ is either $\tau_n$ or $\lfloor \tau_n\rfloor$, and $\tau_n=\tau_n(\gamma,\mu)$, with $\gamma$ chosen above, is as defined in~\eqref{eq:tau}.
\end{lemma}

\begin{assertion}\label{lilhalf2}
	Suppose (A1), (A3), (B1) and either (B21) or (B22). Then, with $\mu\geq 1$ and $q\geq 0$ as defined in \eqref{eq:infim} and \eqref{eq:infim2}, respectively,
	\begin{equation*}
		\limsup_{t\to\infty}(\liminf_{t\to\infty})\frac{X^\ast(t)}{(2(q+1)a(t) \log\log
			a(t))^{1/2}}\geq 1~(\leq -1)\quad\text{{\rm a.s.}}
	\end{equation*}
	and $$\limsup_{t\to\infty}(\liminf_{t\to\infty})\frac{X^\ast(t)}{(2(\mu-1)a(t)\log a(t))^{1/2}}\geq 1~(\leq -1)\quad \text{{\rm a.s.}}$$
	in the cases $\mu=1$ and $\mu>1$, respectively.
\end{assertion}

Proposition \ref{lilhalf1} is a counterpart of Proposition 4.6 in \cite{Buraczewski+Iksanov+Kotelnikova:2023+}.
Although it is tempting to believe that the proof of Proposition 4.6 in~\cite{Buraczewski+Iksanov+Kotelnikova:2023+} goes through as well, this is not the case. First,  
in the proof of \eqref{eq:Yf}, instead of dealing with the process $X$ and its mean $b$ (which are not monotone anymore), we work with their nondecreasing majorants $Y$ and $f$. It is not obvious that $Y$ and $f$ are bounded from above similarly 
to $X$ and $b$. Second, unlike in~\cite{Buraczewski+Iksanov+Kotelnikova:2023+} we do not require that the variance $a$ is asymptotically nondecreasing. Hence, putting $a(t_n)$ in the denominator of \eqref{princip} is not allowed. 
\begin{assertion}\label{lilhalf1}
	Suppose (A1)-(A5). Then, with $\mu\geq 1$ and $q\geq 0$ as defined in \eqref{eq:infim} and~\eqref{eq:infim2}, respectively, 	
	\begin{equation*}
		\limsup_{t\to\infty}(\liminf_{t\to\infty})\frac{X^\ast(t)}{(2(q+1)a(t) \log\log
			a(t))^{1/2}}\leq 1~(\geq -1)\quad \text{\rm a.s.}
	\end{equation*}
	and $$\limsup_{t\to\infty}(\liminf_{t\to\infty})\frac{X^\ast(t)}{(2(\mu-1)a(t)\log a(t))^{1/2}}\leq 1~(\geq -1)\quad \text{{\rm a.s.}}$$
	in the cases $\mu=1$ and $\mu>1$, respectively.
\end{assertion}

\begin{proof}[Proof of Proposition \ref{lilhalf1}]
	In view of (A4), it is enough to show that, for each $\varrho\in (0,1)$ and each positive $\kappa$ sufficiently close to $0$,
	\begin{equation}\label{princip}
		\limsup_{n\to\infty}\frac{\sup_{u\in [t_n,\,t_{n+1}]}\,X^\ast(u)}{(2 a(s_n) h_\varrho(a(s_n)))^{1/2}}\leq 1+\kappa\quad\text{a.s.},
	\end{equation}
	where $t_n=t_n(\kappa,\mu)$ and $s_n=s_n(\kappa,\mu)$ are
	as defined in \eqref{tn} and (A4), respectively, $h_\varrho=g_{1,\,q_\varrho}$ if $\mu$ in \eqref{eq:infim} is equal to $1$ and $h_\varrho=g_{\mu_\varrho}$ if $\mu>1$ (see \eqref{eq:defg} for the definitions of $g_{1,\,q_\varrho}$ and $g_{\mu_\varrho}$). Indeed, if \eqref{princip} holds
	true, then, for large enough $n$, 
	\begin{multline*}
		\limsup_{t\to\infty}\frac{X^\ast(t)}{(2 a(t) h_\varrho (a(t)))^{1/2}}=\limsup_{t\to\infty}\frac{X^\ast(t)}{(2 a_0(t) h_\varrho (a_0(t)))^{1/2}}\\\leq \limsup_{n\to\infty}\frac{\sup_{u\in [t_n,\,t_{n+1}]}\,X^\ast(u)}{(2 a_0(s_n)h_\varrho (a_0(s_n)))^{1/2}}=\limsup_{n\to\infty}\frac{\sup_{u\in [t_n,\,t_{n+1}]}\,X^\ast(u)}{(2 a(s_n)h_\varrho (a(s_n)))^{1/2}}\leq 1+\kappa\quad\text{a.s.}
	\end{multline*}
	The relation $\liminf_{t\to\infty}\frac{X^\ast(t)}{(2 a(t) h_\varrho (a(t)))^{1/2}}\geq -1-\kappa$ a.s.\ does not require a separate proof. It follows from the argument for $\limsup$ upon replacing $\eta_k(t)$ with $-\eta_k(t)$.
	
	To obtain \eqref{princip}, we first prove in Lemma \ref{inter1} that
	\begin{equation}\label{princip1}
		{\limsup}_{n\to\infty}\frac{X^\ast(s_n)}{(2 a(s_n)h_\varrho (a(s_n)))^{1/2}}\leq 1+\kappa\quad \text{{\rm a.s.}}
	\end{equation}
	and then show that
	\begin{equation}\label{princip2}
		\lim_{n\to\infty}\frac{\sup_{u\in [t_n,\, t_{n+1}]}|X^\ast(u)-X^\ast(s_n)|}{(a(s_n)h_\varrho(a(s_n)))^{1/2}}=0\quad\text{a.s.}
	\end{equation}
	
	\begin{lemma}\label{inter1}
		Suppose (A1), (A3) and (A4). Then relation \eqref{princip1} holds for any $\kappa\in (0, (\sqrt{5}-1)/2)$.
	\end{lemma}
	\begin{proof}
		Fix any $\kappa\in (0,(\sqrt{5}-1)/2)$. We first show that there exists $\rho=\rho(\kappa)>0$ satisfying
		\begin{equation}\label{eq:choiceofrho}
			(1-\kappa)(1+\kappa)^2(2-\exp(2(1+\kappa)\rho))>1.
		\end{equation}
		To prove this, note that our choice of $\kappa$ ensures $(1-\kappa)(1+\kappa)^2>1$. Observe next that as
		positive $\rho$ approaches $0$, $2-\exp(2(1+\kappa)\rho)$ becomes arbitrary close to $1$, thereby justifying
		$2-\exp(2(1+\kappa)\rho)>(1-\kappa)^{-1}(1+\kappa)^{-2}$.
		
		By Lemma 4.1 in \cite{Buraczewski+Iksanov+Kotelnikova:2023+}, for $\vartheta\in\mr$ and $t\geq 0$, $$\me\exp(\vartheta X^\ast(t))\leq \exp(2^{-1}\vartheta^2\exp(|\vartheta|)a(t)).$$ Fix any $\theta\in\mr$ and put $\vartheta=\theta/(2a(s_n)h_\varrho(a(s_n)))^{1/2}$. Observe that, for large enough $n$, $h_\varrho(a(s_n))/a(s_n)\leq2 \rho^2$ for $\rho$ satisfying \eqref{eq:choiceofrho}.
		An application of Markov's inequality then yields,
		for large $n$ as above,
		\begin{multline*}
			\mmp\Big\{\frac{X^\ast(s_n)}{(2a(s_n)h_\varrho(a(s_n)))^{1/2}}>1+\kappa\Big\}\leq \eee^{-(1+\kappa)\theta} \me \exp\Big(\theta \frac{X^\ast(s_n)}{(2a(s_n)h_\varrho(a(s_n)))^{1/2}}\Big)\\\leq \exp\Big(-(1+\kappa)\theta+\frac{\theta ^2}{4h_\varrho(a(s_n))}\exp\Big(\frac{\rho |\theta|}{h_\varrho(a(s_n))}\Big)\Big).
		\end{multline*}
		Putting $\theta=2(1+\kappa)h_\varrho(a(s_n))$ and then invoking Lemma \ref{lem:aux1}(a) we obtain
		\begin{multline*}
			\mmp\Big\{\frac{X^\ast(s_n)}{(2a(s_n)h_\varrho(a(s_n)))^{1/2}}>1+\kappa\Big\}\leq \exp(-(1+\kappa)^2(2-\exp(2(1+\kappa)\rho))h_\varrho(a(s_n)))\\=O\Big(\frac{1}{n^{(1-\kappa)(1+\kappa)^2(2-\exp(2(1+\kappa)\rho))}}\Big),\quad n\to\infty.
		\end{multline*}
		According to \eqref{eq:choiceofrho}, $$\sum_{n\geq n_0}\mmp\Big\{\frac{X^\ast(t_n)}{(2a(t_n)h_\varrho(a(t_n)))^{1/2}}>1+\kappa\Big\}<\infty$$ for some $n_0\in\mn$ large enough. An application of the Borel-Cantelli lemma completes the proof of Lemma \ref{inter1}.
	\end{proof}
	
	Next, in order to prove \eqref{princip2} it suffices to show that
	\begin{equation}\label{eq:3}
		\lim_{n\to\infty}\frac{\sup_{u\in [t_n,\, t_{n+1}]}|X^\ast(u)-X^\ast(t_n)|}{(a(s_n))^{1/2}}=0\quad\text{a.s.}
	\end{equation}
	and
	\begin{equation}\label{eq:4}
		\lim_{n\to\infty}\frac{|X^\ast(t_n)-X^\ast(s_n)|}{(a(s_n))^{1/2}}=0\quad\text{a.s.}
	\end{equation}
	Since \eqref{eq:4} is a consequence of \eqref{eq:3}, we are left with proving \eqref{eq:3}.
	
	\noindent {\sc Proof of \eqref{eq:3}.} Let $t_n=t_{0,\,n}<\ldots<t_{j,\,n}=t_{n+1}$ be a partition defined in~(A5). With this at hand, write
	\begin{multline*}
		\sup_{u\in [t_n,\, t_{n+1}]}|X^\ast(u)-X^\ast(t_n)|\\=\max_{1\leq k\leq j}\sup_{v\in
			[0,\,t_{k,\,n}-t_{k-1,\,n}]}|(X^\ast(t_{k-1,\,n})-X^\ast(t_n))+(X^\ast(t_{k-1,\,n}+v)-X^\ast(t_{k-1,\,n}))|\\\leq \max_{1\leq k\leq j}|X^\ast(t_{k-1,\,n})-X^\ast(t_n)|+\max_{1\leq k\leq j}\sup_{v\in
			[0,\,t_{k,\,n}-t_{k-1,\,n}]}|(X^\ast(t_{k-1,\,n}+v)-X^\ast(t_{k-1,\,n}))|\quad\text{a.s.}
	\end{multline*}
	
	By Markov's inequality and Lemma \ref{billappl},
	for any $r>0$ and all $\varepsilon>0$,
	\begin{multline*}
		\mmp\Big\{\max_{1\leq k\leq j}|X^\ast(t_{k-1,\,n})-X^\ast(t_n)|>\varepsilon (a(s_n))^{1/2}\Big\}\\\leq \frac{\me (\max_{1\leq k\leq j}\,|X^\ast(t_{k-1,\,n})-X^\ast(t_n)|)^{2r}}{\varepsilon^{2r}(a(s_n))^r}\leq \frac{A_r(v_{n+1}(\kappa, \mu)-v_n(\kappa, \mu))^r}{\varepsilon^{2r}(a(s_n))^r}.
	\end{multline*}
	By Lemma \ref{lem:aux1}(b), there exists an integer $r\geq 2$ such that the right-hand side forms a sequence which is summable in $n$. Hence, an application of the Borel-Cantelli lemma yields
	\begin{equation*}
		\lim_{n\to\infty}\frac{\max_{1\leq k\leq j}|X^\ast(t_{k-1,\,n})-X^\ast(t_n)|}{(a(s_n))^{1/2}}=0\quad\text{a.s.}
	\end{equation*}
	
	Next, we work towards proving that
	\begin{equation}\label{eq:Yf}
		\lim_{n\to\infty}\frac{\max_{1\leq k\leq j}\sup_{v\in [0,\, t_{k,\,n}-t_{k-1,\,n}]}|X^\ast(t_{k-1,\,n}+v)-X^\ast(t_{k-1,\,n})|}{(a(s_n))^{1/2}}=0\quad \text{a.s.}
	\end{equation}
	According to (A2), for any $0\le s<t$, $|X(t)-X(s)|\le Y(t)-Y(s)$ a.s., where the process $Y$ is a.s.\ nondecreasing.
	Taking into account Remark \ref{rem:bf} we obtain
	\begin{multline*}
		\sup_{v\in [0,\,t_{k,\,n}-t_{k-1,\,n}]}|X^\ast(t_{k-1,\,n}+v)-X^\ast(t_{k-1,\,n})|\\\leq \sup_{v\in [0,\,t_{k,\,n}-t_{k-1,\,n}]}|X(t_{k-1,\,n}+v)-X(t_{k-1,\,n})|+\sup_{v\in [0,\,t_{k,\,n}-t_{k-1,\,n}]}|b(t_{k-1,\,n}+v)-b(t_{k-1,\,n})|\\
		\leq \sup_{v\in [0,\,t_{k,\,n}-t_{k-1,\,n}]}(
		Y(t_{k-1,\,n}+v)-Y(t_{k-1,\,n}))
		+\sup_{v\in [0,\,t_{k,\,n}-t_{k-1,\,n}]}(
		f(t_{k-1,\,n}+v)-f(t_{k-1,\,n}))
		\\= Y(t_{k,\,n})-Y(t_{k-1,\,n})+f(t_{k,\,n})-f(t_{k-1,\,n})\quad\text{a.s.}
	\end{multline*}
	By (A1) and (A5),
	\begin{equation}\label{eq:f}
		\frac{\max_{1\leq k\leq j}\,(f(t_{k,\,n})-f(t_{k-1,\,n}))}{(a(s_n))^{1/2}}\leq \frac{A}{(a(s_n))^{1/2}}~\to 0,\quad n\to\infty.
	\end{equation}
	Finally, for all $\varepsilon>0$,
	\begin{multline}\label{eq:Y}
		\mmp\big\{\max_{1\leq k\leq j}(Y(t_{k,\,n})-Y(t_{k-1,\,n}))>\varepsilon (a(s_n))^{1/2}\big\}\le \sum_{k=1}^{j} \mmp\big\{Y(t_{k,\,n})-Y(t_{k-1,\,n})>\varepsilon (a(s_n))^{1/2}\big\}\\
		\le \eee^{-\varepsilon (a(s_n))^{1/2}} \sum_{k=1}^{j}  \me \eee^{Y(t_{k,\,n})-Y(t_{k-1,\,n})}\le \eee^{-\varepsilon (a(s_n))^{1/2}} \sum_{k=1}^{j} \exp((\eee^M-1)(f(t_{k,\,n})-f(t_{k-1,\,n})))\\\leq \exp(A(\eee^M-1))j \eee^{-\varepsilon (a(s_n))^{1/2}},
	\end{multline}
	having utilized Markov's inequality for the second inequality, Lemma \ref{lem:exp_mom} for the third and (A5) for the fourth. Invoking (A5) once again we conclude that the right-hand side is summable in $n$.
	Hence, an application of the Borel-Cantelli lemma yields
	$$\lim_{n\to\infty}\frac{\max_{1\leq k\leq j}(X(t_{k,\,n})-X(t_{k-1,\,n}))}{(a(s_n))^{1/2}}=0\quad\text{a.s.}$$
	The proofs of both \eqref{eq:3} and Proposition~\ref{lilhalf1} are complete.
\end{proof}

\section{Proofs related to Karlin's occupancy scheme}\label{sec:Karlin}

\subsection{Auxiliary results}

For ease of reference, we state two known results. The former is an obvious extension of Theorem 1.5.3 in \cite{Bingham+Goldie+Teugels:1989}. The latter is Lemma 6.2 in \cite{Buraczewski+Iksanov+Kotelnikova:2023+}.
\begin{lemma}\label{lem:monotone}
	Let $f$ be a function which varies regularly at $\infty$ of positive index and $g$ a positive nondecreasing function with $\lim_{t\to\infty}g(t)=\infty$. Then there exists a nondecreasing function $h$ satisfying
	$f(g(t))\sim h(t)$ as $t\to\infty$.
\end{lemma}

\begin{lemma}\label{eq:thm21}
	(a) Conditions \eqref{eq:deHaan} and \eqref{eq:slowly} entail
	\begin{equation}\label{eq:rho}
		\rho(t)~\sim~(\beta+1)^{-1}(\log t)^{\beta+1}l(\log t),\quad t\to\infty.
	\end{equation}
	
	\noindent (b) Conditions \eqref{eq:deHaan} and \eqref{eq:slowly2} entail
	\begin{equation}\label{eq:rho2}
		\rho(t)~\sim~(\sigma \lambda)^{-1}\exp(\sigma (\log t)^\lambda)(\log t)^{1-\lambda},\quad t\to\infty.
	\end{equation}
\end{lemma}

\subsection{Asymptotic behavior of $\me K_j(t)$ and $\var K_j(t)$}

Given next is a collection of results on the asymptotics of $\me K_j(t)$ and $\var K_j(t)$ taken from Lemma~6.5 of \cite{Buraczewski+Iksanov+Kotelnikova:2023+}. Recall that $\Pi_{\ell,\,\infty}$ denotes the subclass of the de Haan class $\Pi$ with the auxiliary functions $\ell$, see~\eqref{eq:deHaan}, satisfying $\lim_{t\to\infty}\ell(t)=\infty$.
\begin{lemma}\label{lem:old}
	Assume that $\rho\in \Pi_{\ell,\,\infty}$. Then, for each $j\in\mn$,
	\begin{equation}\label{eq:mean0old}
		\me K_j(t)~\sim~ \rho(t),\quad t\to\infty
	\end{equation}
	and
	\begin{equation*}
		{\rm Var}\,K_j(t)~\sim~ \Big(\log 2- \sum_{k=1}^{j-1}\frac{(2k-1)!}{(k!)^2 2^{2k}}\Big)\ell(t),\quad t\to\infty.
	\end{equation*}
	
	Assume that $\rho(t)\sim t^\alpha L(t)$ as $t\to\infty$ for some $\alpha\in(0,1]$ and some $L$ slowly varying at $\infty$. If $\alpha\in(0,1)$ and $j\in\mn$ or $\alpha=1$ and $j\geq 2$, then, as $t\to\infty$,
	\begin{equation}\label{eq:oldmomincreas}
		\me K_j(t)~\sim~\frac{\Gamma(j-\alpha)}{(j-1)!}\rho(t),
	\end{equation}
	and
	\begin{equation}\label{eq:varold}
		\lim_{t\to\infty}\frac{{\rm Var}\, K_j(t)}{\rho(t)}=\Big(\sum_{i=0}^{j-1}\frac{\Gamma(i+j-\alpha)}{i!(j-1)!2^{i+j-1-\alpha}}-\frac{\Gamma(j-\alpha)}{(j-1)!}\Big)>0.
	\end{equation}
	
	If $\alpha=1$, then
	\begin{equation}\label{eq:varalone}
		{\rm Var}\, K_1(t)~\sim~ \me K_1(t)~\sim~ t\hat L(t),\quad t\to\infty.
	\end{equation}
\end{lemma}

\subsection{Asymptotic behavior of $\me K^*_j(t)$ and $\var K^*_j(t)$}

For $j\in\mn$, the asymptotics of $t\mapsto \me K^*_j(t)$ as stated in Theorems \ref{thm:Karlin0}, \ref{thm:Karlin} and \ref{thm:Karlin1} can be found in Lemma 6.5 of \cite{Buraczewski+Iksanov+Kotelnikova:2023+}. Next, we show that, for $j\in\mn$, the functions $t\mapsto \var K^*_j(t)$ exhibit the asymptotics given in the aforementioned theorems.

\begin{lemma}\label{var}
	Assume that $\rho\in \Pi_{\ell,\,\infty}$. Then, for each $j\in\mn$,
	\begin{equation}\label{eq:var0}
		{\rm Var}\,K^*_j(t)~\sim~ \Big(\frac{1}{j}- \frac{(2j-1)!}{(j!)^2 2^{2j}}\Big)\ell(t),\quad t\to\infty.
	\end{equation}
	
	Assume that $\rho(t)\sim t^\alpha L(t)$ as $t\to\infty$ for some $\alpha\in(0,1]$ and some $L$ slowly varying at $\infty$. If $\alpha\in(0,1)$ and $j\in\mn$ or $\alpha=1$ and $j\geq 2$, then, as $t\to\infty$,
	\begin{equation}\label{eq:var}
		\lim_{t\to\infty}\frac{{\rm Var}\, K^*_j(t)}{t^\alpha L(t)}=c_{j,\,\alpha}>0
	\end{equation}
	with $c_{j,\,\alpha}$ as defined in \eqref{eq:cj} and \eqref{eq:cj1}.
	
	If $\alpha=1$, then
	\begin{equation}\label{eq:var1}
		{\rm Var}\, K^*_1(t)~\sim~ t\hat L(t),\quad t\to\infty.
	\end{equation}
\end{lemma}
\begin{proof}
	Assume that $\rho\in\Pi_{\ell,\,\infty}$. Putting $u=v=0$ in formula (11) of \cite{Iksanov+Kotelnikova:2022} we obtain \eqref{eq:var0}.
	
	According to formula (6) in \cite{Gnedin+Hansen+Pitman:2007},
	\begin{equation}\label{eq:equation_var}
		\var K^*_j(t)=\me K^*_j(t)-2^{-2j}\binom{2j}{j}\me K^*_{2j}(2t), \quad t\ge 0, j\in\mn.
	\end{equation}
	Note that \eqref{eq:equation_var} does not require even regular variation assumption on $\rho$.
	
	Assume that $\rho$ is regularly varying at $\infty$ of index $\alpha=1$. We first discuss the properties of the function $\hat L$ stated in Theorem \ref{thm:Karlin1}. By Lemma 3 in \cite{Karlin:1967},
	$\lim_{t\to\infty}t^{-1}\rho(t)=0$
	and $\int_1^\infty y^{-2}\rho(y){\rm d}y\leq 1$. This implies that the function $\hat L(t)=\int_t^\infty y^{-1}L(y){\rm d}y$ is well-defined for large $t$ and thereupon
	$\lim_{t\to\infty}\hat L(t)=0$. According to Proposition 1.5.9b \cite{Bingham+Goldie+Teugels:1989}, $\hat{L}$ is slowly varying at $\infty$ and satisfies~\eqref{eq:Lhat}.
	This in combination with \eqref{eq:meK1}, \eqref{eq:exo} and \eqref{eq:equation_var} entails
	\eqref{eq:var1}.
	
	Assume now $\alpha\in(0,1)$ and $j\in\mn$ or $\alpha=1$ and $j\geq 2$. Then invoking \eqref{eq:equation_var} and either~\eqref{eq:meKgeneral} or \eqref{eq:meK1} we obtain
	$$
	\lim_{t\to\infty}\frac{\var K^*_j(t)}{t^\alpha L(t)}=\lim_{t\to\infty}\frac{\me K^*_j(t)}{t^\alpha L(t)}-2^{-2j}\binom{2j}{j}\lim_{t\to\infty}\frac{\me K^*_{2j}(2t)}{t^\alpha L(t)}=c_{j,\,\alpha}.
	$$
	We are left with showing
	that the constants $c_{j,\,\alpha}$ are positive for $\alpha\in (0,1)$ and $j\in\mn$ and $\alpha=1$ and $j\geq 2$ or equivalently
	$$
	\frac{2^{\alpha}\Gamma(2j-\alpha)}{2^{2j}j!\Gamma(j-\alpha)}<1.
	$$
	This is a consequence of
	$$
	\frac{2^{\alpha}\Gamma(2j-\alpha)}{2^{2j}j!\Gamma(j-\alpha)}<\frac{2\,(2j-1)!}{2^{2j}j!(j-1)!}=\frac{(2j-1)!}{(2j)!!(2j-2)!!}<1,
	$$
	where $(2n)!!:=2\cdot 4\cdot\ldots\cdot (2n)$ for $n\in\mn$. Here, the last inequality is justified with the help of mathematical induction. The proof of Lemma \ref{var} is complete.
\end{proof}

\subsection{Proof of Theorems \ref{thm:Karlin0}, \ref{thm:Karlin} and \ref{thm:Karlin1}}

For $k\in\mn$ and $t\geq 0$, denote by $\pi_k(t)$ the number of balls in box $k$ at time $t$ in the Poissonized version. It has already been mentioned in Section \ref{karlin} that
the thinning property of Poisson processes implies that the processes $(\pi_1(t))_{t\geq 0}$, $(\pi_2(t))_{t\geq 0},\ldots$ are independent. Moreover, for $k\in\mn$, $(\pi_k(t))_{t\geq 0}$ is a Poisson process with intensity $p_k$. As a consequence, both $K_j^*(t)$ and $K_j(t)$ can be represented as the sums of independent indicators
$$
K_j^*(t)=\sum_{k=1}\1_{\{\pi_k(t)=j\}}\quad\text{and}\quad K_j(t)=\sum_{k=1}\1_{\{\pi_k(t)\ge j\}},\quad t\ge 0, j\in\mn.
$$
Hence, it is reasonable to prove the desired LILs for the small counts by applying Theorem~\ref{thm:main}.

As a preparation, we start with a lemma which facilitates
checking condition (B22) of Theorem \ref{thm:main}.

\begin{lemma}\label{lem:karl}
	Assume that either $\rho\in \Pi_{\ell,\,\infty}$ or $\rho$ is regularly varying at $\infty$ of index $\alpha\in(0,1]$. If $\rho\in\Pi_{\ell,\,\infty}$ and $j\in\mn$ or $\alpha\in (0,1)$ and $j\in\mn$ or $\alpha=1$ and $j\geq 2$, then for any positive functions $c$ and $d$ satisfying $\lim_{t\to\infty} c(t)=\infty$, $\lim_{t\to\infty}(c(t)/t)=0$ and $\lim_{t\to\infty}(d(t)/t)=\infty$,
	\begin{equation*}\label{eq:cutvar}
		{\rm Var}\Big(\sum_{k\ge 1} \1_{\{c(t)<1/p_k\leq d(t)\}}\1_{\{\pi_k(t)= j\}}\Big) ~\sim~ {\rm Var}\,K_j^*(t),\quad t\to\infty.
	\end{equation*}
\end{lemma}
\begin{proof}
	We start by proving a simple but an important inequality.
	Since
	\begin{equation*}
		\cov (\1_{\{\pi_k(t)\ge j\}}, \1_{\{\pi_k(t)\ge j+1\}})=\mmp\{\pi_k(t)\ge j+1\}-\mmp\{\pi_k(t)\ge j\}\mmp\{\pi_k(t)\ge j+1\}\ge 0,
	\end{equation*}
	we infer
	\begin{multline*}
		\var(\1_{\{\pi_k(t)=j\}})=\var(\1_{\{\pi_k(t)\ge j\}}-\1_{\{\pi_k(t)\ge j+1\}})\\=\var(\1_{\{\pi_k(t)\ge j+1\}})+\var(\1_{\{\pi_k(t)\ge j\}})-2\cov (\1_{\{\pi_k(t)\ge j\}}, \1_{\{\pi_k(t)\ge j+1\}})\\\le \var\1_{\{\pi_k(t)\ge j+1\}}+\var\1_{\{\pi_k(t)\ge j\}}.
	\end{multline*}
	
	Therefore, it is enough to show that in the setting of the lemma, for all $j\geq 2$ in the case $\alpha=1$ and for all $j\in\mn$ in the other cases,
	\begin{equation}\label{eq:a}
		{\rm Var}\,\Big(\sum_{k\ge 1} \1_{\{1/p_k> d(t)\}}\1_{\{\pi_k(t)\ge j\}}\Big)= o(\var K_j^*(t)), \quad t\to\infty
	\end{equation}
	and
	\begin{equation}\label{eq:b}
		{\rm Var}\,\Big(\sum_{k\ge 1} \1_{\{1/p_k\leq c(t)\}}\1_{\{\pi_k(t)\ge j\}}\Big)=o(\var K_j^*(t)), \quad t\to\infty.
	\end{equation}
	According to formulae (86), (87), (79) and (80) in \cite{Buraczewski+Iksanov+Kotelnikova:2023+},
	\begin{equation}\label{eq:a121}
		{\rm Var}\,\Big(\sum_{k\ge 1} \1_{\{1/p_k> d(t)\}}\1_{\{\pi_k(t)\geq j\}}\Big)= o(\ell(t)), \quad t\to\infty,
	\end{equation}
	\begin{equation}\label{eq:b121}
		{\rm Var}\,\Big(\sum_{k\ge 1} \1_{\{1/p_k\leq c(t)\}}\1_{\{\pi_k(t)\geq j\}}\Big)=o(\ell(t)), \quad t\to\infty,
	\end{equation}
	\begin{equation}\label{eq:aold}
		{\rm Var}\,\Big(\sum_{k\ge 1} \1_{\{1/p_k> d(t)\}}\1_{\{\pi_k(t)\geq j\}}\Big)= o(\rho(t)), \quad t\to\infty
	\end{equation}
	and
	\begin{equation}\label{eq:bold}
		{\rm Var}\,\Big(\sum_{k\ge 1} \1_{\{1/p_k\leq c(t)\}}\1_{\{\pi_k(t)\geq j\}}\Big)=o(\rho(t)), \quad t\to\infty.
	\end{equation}
	In view of \eqref{eq:var0} or \eqref{eq:var}, depending on the setting, relations \eqref{eq:a121} or \eqref{eq:aold}, \eqref{eq:b121} or \eqref{eq:bold} are equivalent to \eqref{eq:a} and \eqref{eq:b}.
\end{proof}

\begin{proof}[Proof of Theorems \ref{thm:Karlin0}, \ref{thm:Karlin} and \ref{thm:Karlin1}]
	We first prove Theorem \ref{thm:Karlin1} in the case $j=1$. This setting is much simpler than the others, for the LIL for
	\begin{equation}\label{eq:equ}
		K^*_1(t)=K_1(t)-K_2(t)
	\end{equation}
	can be derived from the already available LILs for $K_1(t)$ and $K_2(t)$.
	
	The statements of Theorem \ref{thm:Karlin1} concerning the function $\hat L$ has already been justified in the proof of Lemma \ref{var}. According to \eqref{eq:varalone} and \eqref{eq:var1}, $\var K^*_1(t)\sim\var K_1(t)\sim t\hat L(t)$ as $t\to\infty$. Invoking the latter relation, \eqref{eq:varold} and \eqref{eq:Lhat} we conclude that, $\var K_2(t)\sim 2^{-1}tL(t)=o(\var K_1(t))$ as $t\to\infty$.
	By Theorem~3.4 and Remark 1.7 in \cite{Buraczewski+Iksanov+Kotelnikova:2023+}, $$\limsup_{t\to\infty}(\liminf_{t\to\infty})\frac{K_2(t)-\me K_2(t)}{(\var K_2(t)\log\log\var K_2(t))^{1/2}}=2^{1/2}~(-2^{1/2})\quad\text{a.s.}$$
	As a consequence, $K_2(t)-\me K_2(t)=o((\var K_1(t)\log\log\var K_1(t))^{1/2})$ a.s. as $t\to\infty$. Now, in view of \eqref{eq:equ},
	\begin{multline*}
		\limsup_{t\to\infty} (\liminf_{t\to\infty}) \frac{K^*_1(t)-\me K^*_1(t)}{(\var K^*_1(t)\log\log\var K^*_1(t))^{1/2}}\\=\limsup_{t\to\infty} (\liminf_{t\to\infty})\frac{K_1(t)-\me K_1(t)}{(\var K_1(t)\log\log\var K_1(t))^{1/2}}\quad\text{a.s.}
	\end{multline*}
	Armed with this, the claim of Theorem \ref{thm:Karlin1} in the case $j=1$ is secured by Theorem 3.4 in \cite{Buraczewski+Iksanov+Kotelnikova:2023+}. Indeed, the theorem states that depending on whether relation
	\eqref{eq:exotic} holds or not, the right-hand side is either equal to $2^{1/2}$ ($-2^{1/2}$) or is not larger than $2^{1/2}$ (not smaller than $-2^{1/2}$) a.s.

	In the remaining part of the proof we treat simultaneously
	Theorems \ref{thm:Karlin0} and \ref{thm:Karlin} and the case $j\geq 2$ of Theorem \ref{thm:Karlin1}.
	It has already been announced that our plan is to derive the LILs from Theorem \ref{thm:main}.
	Hence, now we work towards checking the conditions of the aforementioned theorem in the present setting.
	
	\noindent {\sc Condition (A1)}
	holds according to \eqref{eq:var0} in conjunction with $\lim_{t\to\infty}\ell(t)=\infty$,
	\eqref{eq:var} and~\eqref{eq:var1}.
	
	\noindent {\sc Condition (A2)} is justified by a representation $\1_{\{\pi_k(t)=j\}}=\1_{\{\pi_k(t)\ge j\}}-\1_{\{\pi_k(t)\ge j+1\}}$ a.s., for all $k,j\in\mn$ and $t\ge 0$, see Remark \ref{rem:A2_ex}. The corresponding function $f$ is given by
	\begin{multline}\label{eq:f_remark}
		f_{j,\,\alpha}(t):=\sum_{k\ge 1}(\mmp\{\pi_k(t)\ge j\}+\mmp\{\pi_k(t)\ge j+1\})=\me K_j(t)+\me K_{j+1}(t)\\=\sum_{k\ge 1} \Big(1-\sum_{i=0}^{j-1} \eee^{-p_kt}\frac{(p_kt)^i}{i!}\Big)+\sum_{k\ge 1} \Big(1-\sum_{i=0}^{j} \eee^{-p_kt}\frac{(p_kt)^i}{i!}\Big).
	\end{multline}
	We bring out the dependence on $j$ and $\alpha$ to distinguish the so defined functions for the different settings. By
	the same reasoning, we write $a_{j,\,\alpha}$ instead of $a$, where $a(t)=\var K_j^*(t)$ for $t\geq 0$.
	
	\noindent {\sc Condition (A3).} Assume first that $\rho\in\Pi_{\ell,\,\infty}$. According to \eqref{eq:mean0old} and \eqref{eq:var0}, for each $j\in\mn$, $f_{j,\,0}(t)\sim 2\rho(t)$ and $a_{j,\,0}(t)\sim C\ell(t)$ as $t\to\infty$, respectively.
	Here and hereafter, $C$ denotes a constant whose
	value is of importance and may vary from formula to formula. Under \eqref{eq:slowly}, invoking \eqref{eq:rho} we conclude
	that~(A3) holds with $\mu=1/\beta+1$. Under \eqref{eq:slowly2}, using~\eqref{eq:rho2} we infer
	$\mu=1$. Thus, we have to check the additional conditions pertaining to the case $\mu=1$. First, the function
	$f_{j,\alpha}$ is continuous. Second, $q=1/\lambda-1$ and $\mathcal{L}(t)\equiv 1$ for all $t\geq 0$ by another appeal to \eqref{eq:rho2}.
	
	Assume now that $\alpha\in (0,1)$ and $j\in\mn$ or $\alpha=1$ and $j\ge 2$. Then, according to \eqref{eq:oldmomincreas} and \eqref{eq:var}, $f_{j,\,\alpha}(t)\sim C a_{j,\,\alpha}(t)$ as $t\to\infty$, which entails $\mu=1$. Further, $f_{j,\,\alpha}$ is continuous,
	$q=0$ and $\mathcal{L}(t)\equiv 1$ for $t\geq 0$.
	
	\noindent {\sc Condition (A4).} Denote by $a_{0;j,\,\alpha}$ a version of $a_0$ for the different settings. Assume first that $\rho\in\Pi_{\ell,\,\infty}$. Then, according to \eqref{eq:var0}, $a_{j,\,0}(t)\sim C \ell(t)$ as $t\to\infty$. Therefore, under~\eqref{eq:slowly}, $a_{0;j,\,0}$ can be chosen as a monotone equivalent of $t\mapsto C(\log t)^{\beta}l(\log t)$ which exists by Lemma \ref{lem:monotone}.
	Under \eqref{eq:slowly2}, $a_{0;j,0}$ can be chosen as $a_{0;j,\,0}(t):=C\exp(\sigma(\log t)^\lambda)$ for all $t\ge 1$.
	
	Assume now that $\alpha\in (0,1)$ and $j\in\mn$ or $\alpha=1$ and $j\ge 2$. Then, according to~\eqref{eq:var}, $a_{0;j,\,\alpha}$ can be chosen as a monotone equivalent of $t\mapsto c_{j,\,\alpha} t^\alpha L(t)$ which exists by
	Lemma~\ref{lem:monotone}.
	
	Thus, in all settings (A4) holds according to Remark \ref{rem:A4}.
	
	\noindent {\sc Condition (A5)} holds according to Remark \ref{suff}, for $f_{j,\,\alpha}$ is continuous and strictly increasing.
	
	\noindent {\sc Condition (B1)} holds in view of
	\begin{equation}\label{eq:small_var}
		a_{j,\,\alpha}(t)=\var K^*_j(t)=\sum_{k\ge 1} \eee^{-p_kt}\frac{(p_kt)^j}{j!}\Big(1-\eee^{-p_kt}\frac{(p_kt)^j}{j!}\Big),\quad t\geq 0
	\end{equation}
	which shows that $a_{j,\,\alpha}$ is a continuous function.
	
	\noindent {\sc Condition (B22).} For $t>1$, put $c(t):=t/\log t$ and $d(t):=t\log t$ and then
	$$R_0(t):= \{k\in\mn: c(t)<1/p_k\leq d(t)\}.
	$$
	By Lemma \ref{lem:karl}, in all settings relation \eqref{eq:one} which is the second part of (B22) holds.
	
	Passing to the first part of (B22), we are going to refer to the table below which contains all the necessary information. In the first line, we list the values of $\mu$ which have already been found while checking (A3). Recall that the definitions of $w_n(\gamma,\mu)$ and $\tau_n$ can be found right after formula \eqref{eq:tau} and in~\eqref{eq:tau}, respectively. 
	\noindent
	\begin{tabular}{|c|c|c|c|}
		\hline
		{\bf Setting} & {\bf $\rho\in\Pi_{\ell,\,\infty}$, \eqref{eq:slowly}} & $\rho\in\Pi_{\ell,\,\infty}$, \eqref{eq:slowly2} & $\alpha\in(0,1)$, $j\in\mn$ or $\alpha=1$, $j\ge 2$\\
		\hline
		$\mu$ & $1/\beta+1$ & $1$ & $1$\\
		\hline
		$w_n(\gamma,\mu)
		$ & $n^{\beta(1+\gamma)}$ & $\exp(n^{\lambda(1+\gamma)})$ & $\exp(n^{1+\gamma})$\\
		\hline
		$\tau_n\sim$ & $\eee^{n^{(1+\gamma)}}o(\eee^{n^{(1+\gamma)}})$ & $\eee^{\sigma^{-1/\lambda}n^{(1+\gamma)}}(1+o(1))$ & $\eee^{\alpha^{-1}n^{(1+\gamma)}}o(\eee^{\alpha^{-1}n^{(1+\gamma)}})$\\
		\hline
	\end{tabular}
	
	\bigskip
	
	\noindent We conclude that in all the settings $\tau_{n+1}/\tau_n$ diverges to $\infty$ superexponentially fast, whereas $\log\tau_n$ only grows polynomially fast. Hence, for large enough $n$, $c(\tau_{n+1})>d(\tau_n)$, which justifies the first part of (B22).
	\color{black}
	
	The proofs of Theorems \ref{thm:Karlin0}, \ref{thm:Karlin} and \ref{thm:Karlin1} are complete.
\end{proof}

\subsection{Proofs of Theorem \ref{thm:depoiss}}

We start with some preparatory work. It is known, see, for instance, Lemma~1 in \cite{Gnedin+Hansen+Pitman:2007}, that for any probability distribution $(p_k)_{k\in\mn}$ and
$j\in\mn$,
\begin{equation}\label{eq:mean_depois}
	\lim_{n\to\infty} |\me\mathcal{K}_j^*(n)-\me K_j^*(n)|=0.
\end{equation}
However, we are not aware of a counterpart of this relation for variances.
Proposition \ref{prop:var_depoiss} fills up this gap. Recall that $\rho\in\Pi_{\ell,\,\infty}$ means that $\rho\in\Pi$ and that its auxiliary function~$\ell$, see \eqref{eq:deHaan}, satisfies
$\lim_{t\to\infty}\ell(t)=\infty$.

\begin{assertion}\label{prop:var_depoiss}
	Assume that either $\rho\in \Pi_{\ell,\,\infty}$ or $\rho$ is regularly varying at $\infty$ of index $\alpha\in (0,1]$. Then, for $j\in\mn$,
	\begin{equation*}\label{eq:ratio}
		\lim_{n\to\infty} \frac{\var \mathcal{K}^*_j(n)}{\var K^*_j(n)}=1.
	\end{equation*}
\end{assertion}

The proof of Proposition \ref{prop:var_depoiss} is partly based on a lemma which is a slight extension of Lemma 6.9 in \cite{Buraczewski+Iksanov+Kotelnikova:2023+}. The new aspect of the lemma is that unlike the cited result it covers the case where $j=1$ and $l\geq 1$ simultaneously.
\begin{lemma}\label{lem:series}
	Assume that either $\rho\in \Pi_{\ell,\,\infty}$ or $\rho$ is regularly varying at $\infty$ of index $\alpha\in(0,1]$.
	Then for $l\ge j$, $l,j\in\mn$,
	\begin{equation}\label{eq:series}
		\sum_{k\ge 1}\binom{n}{l}p_k^l(1-p_k)^{n}=O(\var K_j^*(n)),\quad n\to\infty.
	\end{equation}
	\begin{proof}
		According to the last formula in the proof of Lemma 6.9 in \cite{Buraczewski+Iksanov+Kotelnikova:2023+},
		$$\sum_{k\ge 1}\binom{n}{l}p_k^l(1-p_k)^{n}~\sim~\me K_{l}^*(n+l),\quad n\to\infty.$$
		According to formulae \eqref{eq:meK0}, \eqref{eq:meKgeneral}, \eqref{eq:meK1} and \eqref{eq:exo}, the function $t\mapsto\me K_j^*(t)$ is regularly varying at $\infty$ of index $\alpha\in [0,1]$. This entails $\me K_{l}^*(n+l)\sim\me K_l^*(n)$ as $n\to\infty$.
		
		If $\rho\in \Pi_{\ell,\,\infty}$ or $\rho$ is regularly varying at $\infty$ of index $\alpha\in(0,1)$, then \eqref{eq:series} is a consequence of \eqref{eq:meK0} and \eqref{eq:varK0_1} or \eqref{eq:meKgeneral} and \eqref{eq:varKgeneral}. If $\rho$ is regularly varying at $\infty$ of index $\alpha=1$ and either $j,l\geq 2$ or $j=l=1$, then \eqref{eq:series} follows from~\eqref{eq:meK1} and \eqref{eq:cj1} or \eqref{eq:exo}, respectively. Finally, under the latter regular variation assumption, if $j=1$ and $l\geq 2$, then \eqref{eq:series}, with $o$ replacing $O$, holds true according to~\eqref{eq:meK1},~\eqref{eq:exo} and \eqref{eq:Lhat}.
	\end{proof}
\end{lemma}

\begin{proof}[Proof of Proposition \ref{prop:var_depoiss}]
	
	We start by noting that, in view of
	\eqref{eq:var0}, \eqref{eq:var} or~\eqref{eq:var1}, for
	$j\in\mn$,
	\begin{equation}\label{eq:zero_var}
		\lim_{t\to\infty}\frac{\var K^*_j(t)}{t}=0.
	\end{equation}
	In the case $\alpha=1$ this is secured by $\lim_{t\to\infty}\hat L(t)=0$, which follows from the definition of~$\hat L$.
	
	For $k,j,n\in\mn$, the event $\{\text{the box}~ k~\text{contains exactly}~j~ \text{balls out of}~ n\}$ will be denoted by $A_k(j,n)$.
	Then
	\begin{multline*}
		\var \mathcal{K}_j^*(n)=\sum_{k\ge 1}\mmp(A_k(j,n))(1-\mmp(A_k(j,n)))\\+\sum_{i\neq k}(\mmp(A_i(j,n)\cap A_k(j,n))-\mmp(A_i(j,n))\mmp(A_k(j,n))).
	\end{multline*}
	It is enough to prove that
	\begin{equation}\label{eq:variance}
		\lim_{n\to\infty} \frac{\sum_{k\ge 1}\mmp(A_k(j,n))(1-\mmp(A_k(j,n)))-\var K^*_j(n)}{\var K^*_j(n)}=0
	\end{equation}
	and
	\begin{equation}\label{eq:12}
		\lim_{n\to\infty}\frac{\sum_{i\neq k}(\mmp(A_i(j,n)\cap A_k(j,n))-\mmp(A_i(j,n))\mmp(A_k(j,n)))}{\var K^*_j(n)}=0.
	\end{equation}
	
	\noindent {\sc Proof of \eqref{eq:variance}}. For $k,j,n\in\mn$,
	$$\mmp(A_k(j,n))=\binom{n}{j}p_k^j(1-p_k)^{n-j}.$$ In view of this and \eqref{eq:small_var}, the numerator in \eqref{eq:variance} is equal to
	\begin{equation*}
		\sum_{k\ge 1}\Big(\binom{n}{j}p_k^j(1-p_k)^{n-j}-\eee^{-p_kn}\frac{(p_kn)^j}{j!}-\Big(\binom{n}{j}p_k^{j}(1-p_k)^{n-j}\Big)^2-\Big(\eee^{-p_kn}\frac{(p_kn)^{j}}{j!}\Big)^2\Big).
	\end{equation*}
	According to the penultimate inequality in the proof of Lemma 2.13 in \cite{Iksanov+Kotelnikova:2022}, for large enough~$n$ and any $j\le n$,
	\begin{equation*}\label{eq:ineq1}
		-B_jp_k\le \binom{n}{i}p_k^j(1-p_k)^{n-j}-\eee^{-p_kn}\frac{(p_kn)^j}{j!}\le A_jp_k
	\end{equation*}
	for some positive constants $A_j$ and $B_j$. Therefore,
	$$
	\sum_{k\ge 1} \Big|\binom{n}{j}p_k^j(1-p_k)^{n-j}-\eee^{-p_kn}\frac{(p_kn)^j}{j!}\Big|\le \max(A_j,B_j)=o(\var K^*_j(n)),
	$$
	as $n\to\infty$, since under our assumptions $\lim_{n\to\infty}\var K^*_j(n)=\infty$. Further, write
	\begin{multline*}
		\sum_{k\ge 1} \Big|\Big(\binom{n}{j}p_k^j(1-p_k)^{n-j}\Big)^2-\Big(\eee^{-p_kn}\frac{(p_kn)^j}{j!}\Big)^2\Big|\\=\sum_{k\ge 1} \Big|\Big(\binom{n}{j}p_k^j(1-p_k)^{n-j}-\eee^{-p_kn}\frac{(p_kn)^j}{j!}\Big)\Big|\Big(\binom{n}{j}p_k^j(1-p_k)^{n-j}+\eee^{-p_kn}\frac{(p_kn)^j}{j!}\Big)\\\le 2\sum_{k\ge 1} \Big|\binom{n}{j}p_k^j(1-p_k)^{n-j}-\eee^{-p_kn}\frac{(p_kn)^j}{j!}\Big|=o(\var K^*_j(n)),\quad n\to\infty.
	\end{multline*}
	The proof of \eqref{eq:variance} is complete.
	
	\noindent {\sc Proof of \eqref{eq:12}}. For $k,i,j,n\in\mn$,
	\begin{multline*}
		\mmp(A_i(j,n)\cap A_k(j,n))-\mmp(A_i(j,n))\mmp(A_k(j,n))\\=\binom{n}{j}\binom{n-j}{j}p_i^jp_k^j(1-p_i-p_k)^{n-2j}-\binom{n}{j}\binom{n}{j}p_i^jp_k^j(1-p_i)^{n-j}(1-p_k)^{n-j}\\
		=:C_j(i,k,n).
	\end{multline*}
	We shall use an appropriate decomposition of $C_j$
	\begin{multline*}
		C_j(i,k,n)=\binom{n}{j}\binom{n-j}{j}p_i^jp_k^j\Big((1-p_i-p_k)^{n-2j}-(1-p_i)^{n-j}(1-p_k)^{n-j}\Big)\\-\binom{n}{j}\Big(\binom{n}{j}-\binom{n-j}{j}\Big)p_i^jp_k^j(1-p_i)^{n-j}(1-p_k)^{n-j}=:
		C^{(1)}_j(i,k,n)+C^{(2)}_j(i,k,n).
	\end{multline*}
	To analyze $C^{(1)}_j$ we argue as in the proof of Lemma 1 on p.~152 in \cite{Gnedin+Hansen+Pitman:2007}. Invoking an expansion $$(x-y)^m=x^m+O(mx^{m-1}y),\quad m\to\infty,$$ which holds for positive $x$ and $y$, $x>y$, with $x=(1-p_i)(1-p_k)$, $y=p_ip_k$ and $m=n-2j$, we infer
	\begin{multline*}
		C^{(1)}_{j}(i,k,n)=\binom{n}{j}\binom{n-j}{j}p_i^jp_k^j\Big((1-p_i)^{n-2j}(1-p_k)^{n-2j}\Big(1-(1-p_i)^{j}(1-p_k)^{j}
		\Big)\\+O((n-2j)p_ip_k(1-p_i)^{n-2j-1}(1-p_k)^{n-2j-1})\Big)=: F_j(i,k,n)+G_j(i,k,n).
	\end{multline*}
	
	Next, we intend to show that the contributions of $F_j(i,k,n)$, $G_j(i,k,n)$ and $C^{(2)}_j(i,k,n)$ to the sum are negligible in comparison to
	$\var K^*_j(n)$ as $n\to\infty$.
	
	\noindent {\sc Analysis of $G_j$}.
	With Lemma \ref{lem:series} at hand, we obtain
	\begin{multline*}
		\sum_{i\ne k}\binom{n}{j}\binom{n-j}{j}(n-2j)p_i^{j+1}p_k^{j+1}(1-p_i)^{n-2j-1}(1-p_k)^{n-2j-1}\\\le \sum_{i\ge 1}n\binom{n}{j}p_i^{j+1}(1-p_i)^{n-2j-1}\sum_{k\ge 1}
		\binom{n-j}{j}p_k^{j+1}(1-p_k)^{n-2j-1}\\= O\Big(\var K^*_j(n)\frac{\var K^*_j(n)}{n}\Big)=o(\var K^*_j(n)),\quad n\to\infty
	\end{multline*}
	having utilized \eqref{eq:zero_var} for the last limit relation.
	
	\noindent {\sc Analysis of $F_j$}. For $m\in\mn$ and $x\in [0,1]$, $1-x^m\leq m(1-x)$. Using this with $m=j$ and $x=(1-p_i)(1-p_k)$ we conclude that
	$1-(1-p_i)^j(1-p_k)^j\leq j(p_i+p_k-p_ip_k)\leq j(p_i+p_k)$ and thereupon
	\begin{multline*}
		F_j(i,k,n)\le j\binom{n}{j}\binom{n-j}{j}p_i^jp_k^j(1-p_i)^{n-2j}(1-p_k)^{n-2j}(p_i+p_k)\\
		=:F^{(1)}_j(i,k,n)+F^{(2)}_j(i,k,n).
	\end{multline*}
	Further, invoking Lemma \ref{lem:series} yields
	\begin{multline*}
		0\le\sum_{i\neq k}F^{(1)}_j(i,k,n)\le j\sum_{i\ge 1} \binom{n}{j}p_i^{j+1}(1-p_i)^{n-2j}\sum_{k\ge1}\binom{n-j}{j}p_k^j(1-p_k)^{n-2j}\\=O\Big(\frac{\var{K}^*_{j}(n)}{n}\var{K}^*_j(n)\Big)=o(\var K^*_j(n)),\quad n\to\infty.
	\end{multline*}
	Here, the latter asymptotic relation is a consequence of  \eqref{eq:zero_var}.
	
	The argument for $F_j^{(2)}$ is analogous, and we omit details.
	
	\noindent {\sc Analysis of $C^{(2)}_j$}. Notice that $\binom{n}{j}-\binom{n-j}{j}=O(n^{j-1})$ as $n\to\infty$. Hence, mimicking the argument used for the
	analysis of $F_j^{(1)}$ we conclude that
	\begin{multline*}
		\sum_{i\ne k} |C^{(2)}_j(i,k,n)|
		=O\Big(\sum_{i\ne k}n^{2j-1}p_i^jp_k^j(1-p_i)^{n-j}(1-p_k)^{n-j}\Big)\\=O\Big( \sum_{i\ge 1}n^jp_i^j(1-p_i)^{n-j}\sum_{k\ge 1}n^{j-1}p_k^j(1-p_k)^{n-j}\Big)=O\Big(\var K^*_j(n)\frac{\var K^*_j(n)}{n}\Big)\\=o(\var K_j(n)), \quad n\to\infty.
	\end{multline*}
	
	Combining all the fragments together we arrive at \eqref{eq:12}.
\end{proof}

With Proposition \ref{prop:var_depoiss} at hand, we are ready to prove the LIL stated in Theorem \ref{thm:depoiss}. We argue
along the lines of the proof of Theorem 3.7 in \cite{Buraczewski+Iksanov+Kotelnikova:2023+}.

\begin{proof}[Proof of Theorem \ref{thm:depoiss}.]
	The deterministic and Poissonized schemes discussed in Section~\ref{karlin} are not necessarily defined on a common probability space. Our plan is to deduce LILs for $\mathcal{K}_j^\ast(n)$ from the corresponding LILs for $K_j^\ast(t)$. To this end, we need to {\it couple} the two schemes. Let $X_1$, $X_2,\ldots$ be independent random variables with distribution $(p_k)_{k\in\mn}$, which are independent of a Poisson process $\pi$ and particularly its arrival sequence $(S_n)_{n\in\mn}$. For all $j,n\in\mn$ and $t\geq 0$, we define coupled versions of $\mathcal{K}_j(n)$, $\mathcal{K}_j^*(n)$, $K_j(t)$ and $K_j^\ast(t)$ as follows, keeping the notation for the variables unchanged:
	$$
		\mathcal{K}_j(n)=\#~\text{of distinct values that the variables}~ X_1, X_2, \ldots, X_n
		\text{take at least $j$ times},
	$$
	$$
		\mathcal{K}_j^*(n)=\#~\text{of distinct values that the variables}~ X_1, X_2, \ldots, X_n
		\text{take exactly $j$ times},
	$$
	$$
		K_j(t)=\#~\text{of distinct values that the variables}~ X_1, X_2, \ldots, X_{\pi(t)}
		\text{take at least $j$ times},
	$$
	$$
		K_j^*(t)=\#~\text{of distinct values that the variables}~ X_1, X_2, \ldots, X_{\pi(t)}
		\text{take exactly $j$ times},
	$$
	To justify the construction, observe that the variable $X_i$ can be thought of as the index of a box hit by the $i$th ball. The most important conclusion of the preceding discussion is that, for all $j,n\in\mn$, $\mathcal{K}_j^*(n)=K_{j}^*(S_n)$ a.s. (for the coupled variables).
	
	We prove the result in several steps.
	
	\noindent {\sc Step 1}.
	According to Step 2 of the proof of Theorem 3.7 in \cite{Buraczewski+Iksanov+Kotelnikova:2023+},
	$$\lim_{n\to\infty}\frac{K_{j}(S_n)-K_j(n)}{(\var K_j(n)m(\var K_j(n)))^{1/2}}=0\quad\text{a.s.},$$
	where $m(t)=\log t$ under \eqref{eq:deHaan} and \eqref{eq:slowly} and $m(t)=\log\log t$ under the other assumptions of Theorem \ref{thm:depoiss}.
	
	By Lemmas \ref{lem:old} and \ref{var},
	for $j\in\mn$, $\var K^*_j(t)$ and $\var K_j(t)$ are asymptotically equivalent up to a constant, whence
	$$\lim_{n\to\infty}\frac{K_{j}(S_n)-K_j(n)}{(\var K^*_j(n)m(\var K^*_j(n)))^{1/2}}=0\quad\text{a.s.}$$
	By Lemma \ref{var}, for $j\in\mn$, $\var K^\ast_{j+1}(t)$ and $\var K^*_j(t)$ are asymptotically equivalent up to a constant, unless $\alpha=j=1$. In the latter case, invoking in addition~\eqref{eq:Lhat} we obtain $\var K^*_{j+1}(t)=o(\var K^*_j(t))$ as $t\to\infty$. This in combination with the last centered limit relation, in which we replace $j$ with $j+1$, yields
	$$\lim_{n\to\infty}\frac{K_{j+1}(S_n)-K_{j+1}(n)}{(\var K^*_j(n)m(\var K^*_j(n)))^{1/2}}=0\quad\text{a.s.}$$ Since, for
	$j\in\mn$, $K^*_j(t)=K_j(t)-K_{j+1}(t)$ a.s., subtracting the last two centered limit relations we arrive at
	$$\lim_{n\to\infty}\frac{K_{j}^*(S_n)-K_j^*(n)}{(\var K_j^*(n)m(\var K_j^*(n)))^{1/2}}=0\quad\text{a.s.}$$

	\noindent {\sc Step 2.} Halves of LILs \eqref{eq:LILkar}, \eqref{eq:LILkar1}, \eqref{eq:LILkar2} and \eqref{eq:LILkar200} (\eqref{eq:inf01}, \eqref{eq:inf02}, \eqref{eq:inf1} and \eqref{eq:inf11})
	read
	$$
	\limsup_{n\to\infty}(\liminf_{n\to\infty})\frac{K_j^*(n)-\me K_j^*(n)}{(\var K_j^*(n)m(\var K_j^*(n)))^{1/2}}\le C~ (\ge -C) \quad\text{a.s.},
	$$
	where the case-dependent constant $C$
	is equal to the right-hand side of \eqref{eq:LILkar}, \eqref{eq:LILkar1}, \eqref{eq:LILkar2} or~\eqref{eq:LILkar200} (\eqref{eq:inf01}, \eqref{eq:inf02}, \eqref{eq:inf1} or \eqref{eq:inf11}) , respectively.
	This taken together with the conclusion of Step 1, formula~\eqref{eq:mean_depois} and Proposition \ref{prop:var_depoiss} enables us to obtain
	$$\limsup_{n\to\infty} (\liminf_{n\to\infty}) \frac{\mathcal{K}^\ast_j(n)-\me \mathcal{K}^\ast_j(n)}{(\var \mathcal{K}^\ast_j(n)m(\var \mathcal{K}^\ast_j(n)))^{1/2}}\le C~(\ge -C)\quad \text{a.s.}$$ Here, we have used a decomposition $$\mathcal{K}^\ast_j(n)-\me \mathcal{K}^\ast_j(n)=(K_{j}^*(S_n)-K_j^*(n))+(K_j^\ast(n)-\me K_j^\ast(n))+(\me K_j^\ast(n)-\me \mathcal{K}^\ast_j(n))$$
	a.s. This finishes the proof of $$\limsup_{t\to\infty} (\liminf_{n\to\infty})\frac{\mathcal{K}^*_1(n)-\me \mathcal{K}^*_1(n)}{({\rm Var}\,\mathcal{K}^*_1(n)\log\log {\rm Var}\,\mathcal{K}^*_1(n))^{1/2}}\leq 2^{1/2}~(\geq -2^{1/2})\quad\text{{\rm a.s.}}$$
	in the situation that $\alpha=1$ and relation \eqref{eq:exotic} fails to hold.

	According to Lemma \ref{lem:new},
	for any $\delta>0$ and the deterministic sequence $(\tau_n)$ defined in~\eqref{eq:tau},
	$$\limsup_{n\to\infty} (\liminf_{n\to\infty})\frac{K^*_{j}(\lfloor \tau_n\rfloor)-\me K^*_{j}(\lfloor\tau_n\rfloor)}{C(\var K^*_{j}(\lfloor \tau_n\rfloor)m(\var K^*_j(\lfloor \tau_n\rfloor)))^{1/2}}\ge 1-\delta~(\le -(1-\delta))\quad\text{a.s.}$$
	Combining these inequalities
	with the conclusion of Step 1, formula~\eqref{eq:mean_depois} and Proposition~\ref{prop:var_depoiss} we arrive at
	\begin{multline*}
		\limsup_{n\to\infty}(\liminf_{n\to\infty})\frac{\mathcal{K}^*_{j}(n)-\me \mathcal{K}^*_{j}(n)}{C(\var \mathcal{K}^*_{j}(n)m(\var \mathcal{K}^*_j(n)))^{1/2}}\\\geq (\leq) \limsup_{n\to\infty} (\liminf_{n\to\infty})\frac{\mathcal{K}^*_{j}(\lfloor \tau_n\rfloor)-\me \mathcal{K}^*_{j}(\lfloor \tau_n\rfloor)}{C(\var \mathcal{K}^*_{j}(\lfloor \tau_n\rfloor)m(\var \mathcal{K}^*_j(\lfloor \tau_n\rfloor)))^{1/2}} \ge 1-\delta ~ (\le -(1-\delta))\quad \text{a.s.}
	\end{multline*}
	Sending $\delta\to 0+$ yields
	$$\limsup_{n\to\infty}(\liminf_{n\to\infty})\frac{\mathcal{K}^*_{j}(n)-\me \mathcal{K}^*_{j}(n)}{(\var \mathcal{K}^*_{j}(n)m(\var \mathcal{K}^*_j(n)))^{1/2}}\geq C~ (\le -C)\quad\text{a.s.},$$
	which finishes the proof.
\end{proof}

\noindent {\bf Acknowledgement}. The research was supported by Applied Probability Trust in the framework of a Ukraine Support Scheme.

\end{document}